\documentclass[final,3p,times]{elsarticle}

\usepackage{amsmath}
\usepackage{amsfonts}
\usepackage{amssymb}
\usepackage[english]{babel}
\usepackage{color}
\usepackage{graphicx}
\usepackage{lmodern}
\usepackage{amsthm}
\usepackage{mathrsfs}
\usepackage{microtype}
\usepackage{mathscinet}
\usepackage{enumitem}
\usepackage[cal=boondoxo,bb=ams]{mathalfa}
\usepackage{hyperref}
\hypersetup{hidelinks}

\newtheorem{thm}{Theorem}[section]
 
 \newtheorem{lem}[thm]{Lemma}
 \newtheorem{prop}[thm]{Proposition}
 \newtheorem{defn}[thm]{Definition}
 \newtheorem{rem}[thm]{Remark}

\DeclareMathOperator\supp{supp}

\DeclareMathOperator\Str{Str}
\DeclareMathOperator\Gla{Gla}

\DeclareMathOperator\loc{loc}

\DeclareMathOperator\mix{mix}

\begin{document}

\begin{frontmatter}



\title{Nonexistence of global solutions for a weakly coupled system of semilinear damped wave equations in the scattering case with mixed nonlinear terms}


\author[Frei]{Alessandro Palmieri}
\ead{alessandro.palmieri.math@gmail.com}

\author[Toho]{Hiroyuki Takamura}
\ead{hiroyuki.takamura.a1@tohoku.ac.jp}

\address[Frei]{Institute of Applied Analysis, Faculty for Mathematics and Computer Science, Technical University Bergakademie Freiberg, Pr\"{u}ferstra{\ss}e 9, 09596, Freiberg, Germany}

\address[Toho]{Mathematical Institute, Tohoku University, Aoba, Sendai 980-8578, Japan}

\begin{abstract}

In this paper we consider the blow-up of solutions to a weakly coupled system of semilinear damped wave equations in the scattering case with nonlinearities of mixed type, namely, in one equation a power nonlinearity and in the other a semilinear term of derivative type. The proof of the blow-up results is based on an iteration argument. As expected, due to the assumptions on the coefficients of the damping terms, we find as critical curve in the $p$ - $q$ plane for the pair of exponents $(p,q)$ in the nonlinear terms the same one found by Hidano-Yokoyama and, recently, by Ikeda-Sobajima-Wakasa for the weakly coupled system of semilinear wave equations with the same kind of nonlinearities. In the critical and not-damped case we provide a different approach from the test function method applied by Ikeda-Sobajima-Wakasa 
to prove the blow-up of the solution on the critical curve, improving in some cases the upper bound estimate for the lifespan. More precisely, we combine an iteration argument with the so-called slicing method to show the blow-up dynamic of a weighted version of the functionals used in the subcritical case.
\end{abstract}

\begin{keyword}
Semilinear weakly coupled system; Damped wave equation; Blow-up; Scattering producing damping; Critical curve; Mixed nonlinearities. 


\MSC[2010] Primary 35L71 \sep 35B44; Secondary   35G50 \sep 35G55 \sep 35L05


\end{keyword}

\end{frontmatter}


\section{Introduction}

In this paper we consider a weakly coupled system of wave equations with time-dependent and scattering producing damping terms and with mixed kinds of power nonlinearity, namely, 
\begin{align}\label{weakly coupled system}
\begin{cases}
u_{tt}-\Delta u +b_1(t)u_t = |v|^q,  & x\in \mathbb{R}^n, \ t>0,  \\
v_{tt}-\Delta v +b_2(t) v_t = |\partial_t u|^p,  & x\in \mathbb{R}^n, \ t>0, \\
 (u,u_t,v,v_t)(0,x)= (\varepsilon u_0, \varepsilon u_1, \varepsilon v_0, \varepsilon v_1)(x) & x\in \mathbb{R}^n,
\end{cases}
\end{align}
where $b_1,b_2\in \mathcal{C}([0,\infty))\cap L^1([0,\infty))$ are nonnegative functions, $\varepsilon$ is a positive parameter describing the size of initial data and $p,q>1$. More precisely, we will focus on blow-up phenomena for local solutions and we will derive the corresponding upper bound for the lifespan.

%

In order to motivate the study of \eqref{weakly coupled system}, let us recall some semilinear models which are strongly related to this weakly coupled system.

Let us begin with the Cauchy problem for the semilinear wave equation with power nonlinearity
\begin{align}\label{semilinear wave equation}
\begin{cases}
u_{tt}-\Delta u = |u|^p,  & x\in \mathbb{R}^n, \ t>0,  \\
 (u,u_t,)(0,x)= (\varepsilon u_0, \varepsilon u_1)(x), & x\in \mathbb{R}^n.
\end{cases}
\end{align} After John's pioneering paper \cite{John79}, it was conjectured by Strauss in \cite{Str81} that the critical exponent for the Cauchy problem \eqref{semilinear wave equation} is the positive root of the quadratic equation $$(n-1)p^2-(n+1)p-2=0,$$ which is nowadays named after him Strauss exponent and denoted in this paper by $p_{\Str}(n)$. In the classical works 
\cite{Kato80,Glas81,Glas81B,Sid84,Scha85,LinSog95, Geo97,Tat01,Jiao03,YZ06,Zhou07} this conjecture is proved to be true. Here, critical exponent means that for $1<p\leqslant p_{\Str}(n)$ local in time solutions blow up in finite times under certain sign assumptions on the initial data and regardless of the smallness of these, while for $p>p_{\Str}(n)$ the global in time existence of small data solutions holds in suitable function spaces.  Moreover, the sharp lifespan estimate for local solutions has been derived both in the subcritical case and in the critical case, cf. 
\cite{Sid84,Lin90,Zhou92,Zhou93,LinSog96,TakWak11,ZH14,Tak15}.

A similar situation has been studied in the case of the Cauchy for the semilinear wave equation of derivative type as well, namely,
\begin{align}\label{semilinear wave equation Glassey}
\begin{cases}
u_{tt}-\Delta u = |\partial_t u|^p,  & x\in \mathbb{R}^n, \ t>0,  \\
 (u,u_t,)(0,x)= (\varepsilon u_0, \varepsilon u_1)(x), & x\in \mathbb{R}^n.
\end{cases}
\end{align} For \eqref{semilinear wave equation Glassey} it has been proved that the critical exponent is the so-called Glassey exponent $p_{\Gla}(n)\doteq \frac{n+1}{n-1}$, although the global in time existence in the supercritical case for non radial solutions is still open for spatial dimensions $n\geqslant 4$, see also 
\cite{Joh81,Mas83,Sch85,Ram87,Age91,Zhou01} for the blow-up results and \cite{Sid83,HT95,Tzv98,HWY12} for the global existence results.

Concerning the weakly coupled systems of semilinear wave equations
\begin{align}\label{weakly coupled system wave eq}
\begin{cases}
u_{tt}-\Delta u  = G_1(v,\partial_t v),  & x\in \mathbb{R}^n, \ t>0,  \\
v_{tt}-\Delta v  = G_2(u,\partial_t u),  & x\in \mathbb{R}^n, \ t>0, \\
 (u,u_t,v,v_t)(0,x)= (\varepsilon u_0, \varepsilon u_1, \varepsilon v_0, \varepsilon v_1)(x), & x\in \mathbb{R}^n,
\end{cases}
\end{align}
the cases $G_1(v,\partial_t v)=|v|^p$, $G_2(u,\partial_t u)=|u|^q$ and $G_1(v,\partial_t v)=|\partial_t v|^p$, $G_2(u,\partial_t u)=|\partial_t u|^q$ have been studied in \cite{DGM,DelS97,DM,AKT00,KT03,Kur05,GTZ06,KTW12} and in 
\cite{Deng99,Xu04,KKS06,ISW18}, respectively. While in the case of power nonlinearities (that is, for $G_1(v,\partial_t v)=|v|^p$, $G_2(u,\partial_t u)=|u|^q$) the critical curve is given by $$\max\left\{\frac{p+2+q^{-1}}{pq-1},\frac{q+2+p^{-1}}{pq-1}\right\}=\frac{n-1}{2},$$ the case of semilinear terms of derivative type (that is, for $G_1(v,\partial_t v)=|\partial_t v|^p$, $G_2(u,\partial_t u)=|\partial_t u|^q$) the critical curve is
 $$\max\left\{\frac{p+1}{pq-1},\frac{q+1}{pq-1}\right\}=\frac{n-1}{2},$$ even though the global existence part has been studied so far only in the three dimensional and radial symmetric case. Recently, the case with mixed nonlinear terms $G_1(v,\partial_t v)=|v|^q$, $G_2(u,\partial_t u)=|\partial_t u|^p$ has been investigated for \eqref{weakly coupled system wave eq} in \cite{HidYok16,ISW18}. In this paper we shall prove that the for same range of exponents $p,q>1$ as in \cite{ISW18} a blow-up result can be proved in the subcritiacal case even when we add as lower order terms in the linear part damping terms with time-dependent and scattering producing coefficients (see \cite{Wirth04,Wirth06,Wirth07} for this classification of a damping term with time-dependent coefficient for wave models). Furthermore, the same upper bound for the lifespan can be derived. In the critical case, we will restrict our considerations to the not-damped case, improving in some cases the upper bound for the lifespan with respect to \cite{ISW18}, but using a quite different method.
 
 Recently, several results for semilinear wave equations and for weakly coupled systems of semilinear wave equations have been proved in presence of time-dependent and scattering-producing coefficients for damping terms by Lai-Takamura, Wakasa-Yordanov and Palmieri-Takamura. More precisely, the blow-up dynamic for local solutions of 
 \begin{align*}
\begin{cases}
u_{tt}-\Delta u +b(t)u_t = G(u,\partial_t u),  & x\in \mathbb{R}^n, \ t>0,  \\
 (u,u_t)(0,x)= (\varepsilon u_0, \varepsilon u_1)(x), & x\in \mathbb{R}^n,
\end{cases}
\end{align*} 
has been considered in \cite{LT18Scatt,WakYor18damp} for the case of power nonlinearity $G(u,\partial_t u)=|u|^p$, in \cite{LT18Glass} for the case of derivative type $G(u,\partial_t u)=|\partial_t u|^p$ and in \cite{LT18ComNon} for the case of combined nonlinearity $G(u,\partial_t u)=|\partial_t u|^p+|u|^q$.
Finally, really recently the weakly coupled system of semilinear damped wave equations in the scattering case
\begin{align*}
\begin{cases}
u_{tt}-\Delta u +b_1(t)u_t = G_1(v,\partial_t v),  & x\in \mathbb{R}^n, \ t>0,  \\
v_{tt}-\Delta v +b_2(t) v_t = G_2(u,\partial_t u),  & x\in \mathbb{R}^n, \ t>0, \\
 (u,u_t,v,v_t)(0,x)= (\varepsilon u_0, \varepsilon u_1, \varepsilon v_0, \varepsilon v_1)(x), & x\in \mathbb{R}^n,
\end{cases}
\end{align*} has been considered in \cite{PalTak19} for the case with power nonlinearities $G_1(v,\partial_t v)=|v|^p,G_2(u,\partial_t u)=|u|^q$ and in \cite{PalTak19dt} for the  case with semilinear terms of derivative type $G_1(v,\partial_t v)=|\partial_t v|^p,G_2(u,\partial_t u)=|\partial_t u|^q$.

In this paper our approach is based on the following methods: in the subcritical case we employ two multipliers, that are introduced in \cite{LT18Scatt}, in order to apply a standard iteration argument based on lower bound estimates for the spatial integrals of the nonlinear terms and on a coupled system of ordinary integral inequalities; in the critical case, we modify the approach introduced by Wakasa-Yordanov in \cite{WakYor18,WakYor18damp} and adapted to weakly coupled systems in \cite{PalTak19} with the purpose to deal with the nonlinear term of derivative type. We underline that in the case with time-dependent coefficients for the damping terms in the scattering case, we may not apply the revised test function method which is introduced by Ikeda-Sobajima-Wakasa in \cite{ISW18} for semilinear wave models. Furthermore, in the critical case, where we consider the not-damped case as in Section 9 of  \cite{ISW18}, it is interesting to compare how our different approach leads to different upper bound estimates for the lifespan and in some cases to an improvement of these estimates. We refer to \cite{IS18} and to  \cite{IS17,ISW18,PT18,Pal19} for further details on this revised test function method based on a family of self similar solutions of the adjoint linear equation involving Gauss hypergeometric functions, for the study of semilinear heat, Schr\"odinger and damped wave equations and for the treatment of semilinear and scale-invariant model with time-dependent coefficients, respectively.

Before stating the blow-up results of this paper, let us introduce a suitable notion of energy solutions.

\begin{defn} \label{def energ sol intro} Let $u_0,v_0\in H^1(\mathbb{R}^n)$ and $u_1,v_1\in L^2(\mathbb{R}^n)$.
We say that $(u,v)$ is an energy solution of \eqref{weakly coupled system} on $[0,T)$ if 
\begin{align*}
& u\in \mathcal{C}([0,T),H^1(\mathbb{R}^n))\cap \mathcal{C}^1([0,T),L^2(\mathbb{R}^n))\quad \mbox{and} \quad  \partial_t u\in L^p_{\loc}([0,T)\times\mathbb{R}^n), \\
& v\in \mathcal{C}([0,T),H^1(\mathbb{R}^n))\cap \mathcal{C}^1([0,T),L^2(\mathbb{R}^n)) \quad \mbox{and} \quad  v\in  L^q_{\loc}([0,T)\times\mathbb{R}^n)
\end{align*}
satisfy $u(0,x)=\varepsilon u_0(x)$, $v(0,x)=\varepsilon v_0(x)$ in $H^1(\mathbb{R}^n)$ and the equalities
\begin{align} 
\int_{\mathbb{R}^n} & \partial_t u(t,x)\phi(t,x)\,dx-\int_{\mathbb{R}^n}\varepsilon u_1(x)\phi(0,x)\,dx - \int_0^t\int_{\mathbb{R}^n} \partial_t u(s,x)\phi_s(s,x) \,dx\, ds\notag \\
& \ +\int_0^t\int_{\mathbb{R}^n}\nabla u(s,x)\cdot\nabla\phi(s,x)\, dx\, ds+\int_0^t\int_{\mathbb{R}^n}b_1(s)\partial_t u(s,x) \phi(s,x)\,dx\, ds  \notag \\
& =\int_0^t \int_{\mathbb{R}^n}|v(s,x)|^q\phi(s,x)\,dx \, ds  \label{def u}
\end{align} and 
\begin{align} 
\int_{\mathbb{R}^n} & \partial_t v(t,x)\psi(t,x)\,dx-\int_{\mathbb{R}^n}\varepsilon v_1(x)\psi(0,x)\,dx - \int_0^t\int_{\mathbb{R}^n} \partial_t v(s,x)\psi_s(s,x) \,dx\, ds\notag \\
& \ +\int_0^t\int_{\mathbb{R}^n}\nabla v(s,x)\cdot\nabla\psi(s,x)\, dx\, ds+\int_0^t\int_{\mathbb{R}^n} b_2(s) \partial_t v(s,x)\psi(s,x)\,dx\, ds  \notag \\
& =\int_0^t \int_{\mathbb{R}^n}|\partial_t u(s,x)|^p\psi(s,x)\,dx \, ds  \label{def v}
\end{align}
for any test functions $\phi,\psi \in \mathcal{C}_0^\infty([0,T)\times\mathbb{R}^n)$ and any $t\in [0,T)$.
\end{defn}

 Performing a further step of integrations by parts in \eqref{def u}, \eqref{def v} and letting $t\rightarrow T$, we find that $(u,v)$ fulfills the definition of weak solution to \eqref{weakly coupled system}.
 
Let us state the blow-up result for \eqref{weakly coupled system} in the subcritical case.

\begin{thm}\label{Thm blowup |vt|^p, |u|^q} Let $b_1,b_2$ be continuous, nonnegative and summable functions. Let us consider $p,q>1$ satisfying 
\begin{align}\label{critical exponent wave like case system}
\max\left\{\frac{q+1+p^{-1}}{pq-1},\frac{2+q^{-1}}{pq-1}\right\}> \frac{n-1}{2}\,.
\end{align} 

Assume that $u_0,v_0 \in H^1(\mathbb{R}^n)$ and $u_1,v_1\in  L^2(\mathbb{R}^n)$ are nonnegative and compactly supported in $B_R$ 
functions such that $u_1\not \equiv 0$ and $v_0\not \equiv 0$. 

Let $(u,v)$ be an energy solution of \eqref{weakly coupled system} with lifespan $T=T(\varepsilon)$ such that
\begin{align} \label{support condition solution}
\supp u, \, \supp v \subset\{(t,x)\in  [0,T)\times\mathbb{R}^n: |x|\leqslant t+R \}.
\end{align} Then, there exists a positive constant $\varepsilon_0=\varepsilon_0(u_0,u_1,v_0,v_1,n,p,q,b_1,b_2,R)$ such that for any $0<\varepsilon\leqslant\varepsilon_0$ the solution $(u,v)$ blows up in finite time. Moreover,
 the upper bound estimate for the lifespan
\begin{align} \label{lifespan upper bound estimate}
T(\varepsilon)\leqslant C \varepsilon ^{-\max\left\{\Theta_1(n,p,q),\Theta_2(n,p,q)\right\}^{-1}}
\end{align} holds, where C is an independent of $\varepsilon$, positive constant and 
\begin{align}\label{def Lambda(n,p,q) function}
\Theta_1(n,p,q)\doteq \frac{q+1+p^{-1}}{pq-1}-\frac{n-1}{2} \quad \mbox{and} \quad \Theta_2(n,p,q)\doteq\frac{2+q^{-1}}{pq-1}-\frac{n-1}{2} .
\end{align}
\end{thm}

\begin{rem} The upper bound estimates \eqref{lifespan upper bound estimate} for the lifespan coincide with the ones for the case  $b_1=b_2=0$, for more details see also \cite[Section 9]{ISW18}.
\end{rem}

The main result in the critical and not-damped case is the following theorem.

\begin{thm}\label{Thm crit case}
Let $n\geqslant 2$ and $b_1=b_2=0$. Let us assume that  $p,q>1$ satisfy 
\begin{align}\label{critical exponent wave like case system crit}
\max\left\{\frac{q+1+p^{-1}}{pq-1},\frac{2+q^{-1}}{pq-1}\right\}= \frac{n-1}{2}\, ,
\end{align} 
Assume that $u_0,v_0 \in H^1(\mathbb{R}^n)$ and $u_1,v_1\in  L^2(\mathbb{R}^n)$ are nonnegative and compactly supported in $B_R$ 
functions such that $u_1\not \equiv 0$ and $v_0\not \equiv 0$. 
Let $(u,v)$ be a weak solution of 
\begin{align}\label{weakly coupled system crit}
\begin{cases}
u_{tt}-\Delta u  = |v|^q,  & x\in \mathbb{R}^n, \ t>0,  \\
v_{tt}-\Delta v  = |\partial_t u|^p,  & x\in \mathbb{R}^n, \ t>0, \\
 (u,u_t,v,v_t)(0,x)= \varepsilon ( u_0,  u_1, v_0,  v_1)(x) & x\in \mathbb{R}^n,
\end{cases}
\end{align} satisfying \eqref{support condition solution} with lifespan $T=T(\varepsilon)$ (cf. Definition \ref{def weak sol crit case}). 

Then, there exists a positive constant $\varepsilon_0=\varepsilon_0(u_0,u_1,v_0,v_1,n,p,q,R)$ such that for any $\varepsilon\in (0,\varepsilon_0]$ the solution $(u,v)$ blows up in finite time. Moreover,
 the upper bound estimates for the lifespan
\begin{align} \label{lifespan upper bound estimate crit}
T(\varepsilon)\leqslant \begin{cases} \exp\left( C \varepsilon^{-p(pq-1)}\right) \vphantom{\Big(}   & \mbox{if} \ \ \Theta_1(n,p,q)=0,  \\
\exp\left( C \varepsilon^{-q(pq-1)}\right) \vphantom{\Big(}  & \mbox{if} \ \ \Theta_2(n,p,q)=0,  \\
\exp\left(C \varepsilon ^{-\frac{q}{q+1}(pq-1)}\right) & \mbox{if} \ \ \Theta_1(n,p,q)=\Theta_2(n,p,q)=0,\end{cases}
\end{align} hold, where C is an independent of $\varepsilon$, positive constant.
\end{thm}

The remaining part of this paper is organized as follows: in Section \ref{Section iteration frame} we derive the coupled system of ODIs (ordinary differential inequalities) that the spatial averages of the components of a local solution has to satisfy, then, using a suitable pair of multipliers $(m_1,m_1)$ (cf. \eqref{def multipliers} below) we derive the corresponding integral iteration frame from this system of ODIs; in Section \ref{Section lower bounds nonlinear term} we prove suitable lower bounds for the space integrals of the nonlinearities, that is, for $\|\partial_t u(t,\cdot)\|_{L^p(\mathbb{R}^n)}^p,\|v(t,\cdot)\|_{L^q(\mathbb{R}^n)}^q$; hence, in Section \ref{Section iteration procedure} we combine the results from Sections \ref{Section iteration frame}-\ref{Section lower bounds nonlinear term} in an iterative procedure which allows us to determine a sequence of lower bound estimates for the above cited spatial averages; finally, in Section \ref{Section end of the proof of the Theorem subcritical case} we conclude the proof of Theorem \ref{Thm blowup |vt|^p, |u|^q} proving the blow-up result thanks to the sequence of lower bounds obtained via the iteration argument and deriving the upper bound for the lifespan of a local solution. Finally, in Section \ref{Section critical case} we prove Theorem \ref{Thm crit case}. The intermediate steps are similar to the ones for the subcritical case: derivation of the iteration frame, lower bound estimates for integrals related to the nonlinear terms, yet containing a logarithmic factor, and iteration procedure combined with the slicing method. Nevertheless, a crucial difference consists in the choice of the functionals, whose blow-up dynamic is considered. Indeed, differently from the subcritical case, we do not consider  spatial averages of the components of a local solution rather weighted spatial averages of this components.

\subsection*{Notations} Throughout this paper we will use the following notations: $B_R$ denotes the ball around the origin with radius $R$; $f \lesssim g$ means that there exists a positive constant $C$ such that $f \leqslant Cg$ and, analogously, for $f\gtrsim g$; moreover, $f\asymp g$ means $f\lesssim g$ and $f\gtrsim g$; finally, as in the introduction, $p_{\Str}(n)$ and $p_{\Gla}(n)$ denote the Strauss exponent and the Glassey exponent, respectively.

\section{Iteration frame} \label{Section iteration frame}

Let us recall the definition of some multipliers related to our model, which have been introduced in \cite{LT18Scatt}, and some properties of them, that we will employ throughout the remaining sections.
 
\begin{defn} Let $b_1,b_2\in \mathcal{C}([0,\infty))\cap L^1([0,\infty))$ be the nonnegative, time-dependent coefficients in \eqref{weakly coupled system}. We define the \emph{multipliers}
\begin{align}
m_j(t)\doteq \exp \bigg(-\int_t^\infty b_j(\tau) d\tau \bigg) \qquad \mbox{for} \ \  t\geqslant 0 \ \ \mbox{and} \ j=1,2.  \label{def multipliers}
\end{align}
\end{defn}

As $b_1,b_2$ are nonnegative functions, then, $m_1,m_2$ are increasing functions. Moreover, due to the summability of $b_1,b_2$, the multipliers are bounded and 
\begin{align}
m_j(0)\leqslant m_j(t) \leqslant 1 \qquad \mbox{for} \ \  t\geqslant 0 \ \ \mbox{and} \ j=1,2. \label{boundedness multipliers}
\end{align}
Finally, a remarkable property of these multipliers is the following one:
\begin{align}
m_j'(t) = b_j(t) \, m(t)  \qquad \mbox{for} \ \   j=1,2. \label{derivative multiplier}
\end{align} The properties given in \eqref{boundedness multipliers} and \eqref{derivative multiplier} are essential in order to handle and somehow to ``neglect'' the damping term.

Henceforth, we assume that $u_0,u_1,v_0,v_1$ satisfy the assumptions of Theorem \ref{Thm blowup |vt|^p, |u|^q}. Let $(u,v)$ be an energy solution of \eqref{weakly coupled system} on $[0,T)$ in the sense of Definition \ref{def energ sol intro}. Then, we introduce the following pair of functionals
\begin{align}
U(t)\doteq \int_{\mathbb{R}^n} u(t,x) \, dx, \quad V(t)\doteq \int_{\mathbb{R}^n} v(t,x) \, dx. \label{def U,V}
\end{align}
Let us point out that the pair of functionals whose dynamic will investigated in Section \ref{Section iteration procedure} is actually $(U',V)$ due the nonlinearity in \eqref{weakly coupled system}.

The support condition \eqref{support condition solution} can be rewritten as $$\supp u(t,\cdot),\, \supp v(t,\cdot)\subset B_{R+t}\quad \mbox{for any} \  t\geqslant 0.$$
Therefore, using Green's identity, it results that $U,V$ satisfy
\begin{align}
U''(t)+b_1(t)U'(t)&=\int_{\mathbb{R}^n}|v(t,x)|^q\, dx, \label{ODE U} \\
V''(t)+b_2(t)V'(t)&=\int_{\mathbb{R}^n}|\partial_t u(t,x)|^p\, dx. \label{ODE V}
\end{align}

Let us derive first  integral lower bound estimates for $V$ from \eqref{ODE V}.
Multiplying both sides of \eqref{ODE V} by $m_2$ and using \eqref{derivative multiplier}, we get 
\begin{align*}
m_2(t)V''(t)+m_2(t)b_2(t)V'(t)= \frac{d}{dt} \big(m_2(t) V'(t)\big)=m_2(t) \int_{\mathbb{R}^n} |\partial_t u(t,x)|^p\, dx.
\end{align*} Hence, integrating over$[0,t]$  the last relation and rearranging the resulting equation, we have
\begin{align*}
V'(t) &=\frac{m_2(0)}{m_2(t)}V'(0)+\int_0^t \frac{m_2(s)}{m_2(t)} \int_{\mathbb{R}^n} |\partial_t u(s,x)|^p\, dx\, ds \\
&\geqslant m_2(0) V'(0)+ m_2(0)\int_0^t \int_{\mathbb{R}^n} |\partial_t u(s,x)|^p\, dx\, ds,
\end{align*} where in the last step we used \eqref{boundedness multipliers}. A further integration over $[0,t]$ provides
\begin{align}
V(t) & \geqslant V(0)+ m_2(0) V'(0) t+ m_2(0)\int_0^t \int_0^s \int_{\mathbb{R}^n} |\partial_t u(\tau,x)|^p\, dx \, d\tau\, ds \qquad \mbox{for any} \ t\geqslant 0. \label{ODI U, source |vt|^p}
\end{align} Using again the support property for $u_t(t,\cdot)$ and H\"older's inequality, we find that \eqref{ODI U, source |vt|^p} implies
\begin{align}
V(t) & \geqslant C \int_0^t \int_0^s (1+\tau)^{-n(p-1)} (U'(\tau))^p \, d\tau\, ds \qquad \mbox{for any} \ t\geqslant 0. \label{ODI U, coupled with V}
\end{align}
for a suitable positive constant $C=C(n,p,b_2,R)$.

Proceeding in a similar way, we derive now two lower bound estimates for $U'$. A multiplication by $m_1$ in \eqref{ODE U} and a successive integration over $[0,t]$ lead to  
\begin{align*}
U'(t) &=\frac{m_1(0)}{m_1(t)}U'(0)+\int_0^t \frac{m_1(s)}{m_1(t)} \int_{\mathbb{R}^n} |v(s,x)|^q\, dx\, ds .
\end{align*}
Employing again \eqref{boundedness multipliers}, from the last estimate  we derive
\begin{align}
U'(t) &\geqslant m_1(0) U'(0)+ m_1(0)\int_0^t \int_{\mathbb{R}^n} |v(s,x)|^q\, dx\, ds. \label{ODI V, source |u|^q}
\end{align} Finally, thanks to the support condition for $v(t,\cdot)$, by H\"older's inequality we find
\begin{align}
U'(t) & \geqslant K \int_0^t (1+s)^{-n(q-1)} (V(s))^q \, ds \qquad \mbox{for any} \ t\geqslant 0. \label{ODI V, coupled with U}
\end{align} for a suitable positive constant $K=K(n,q,b_1,R)$.

In Section \ref{Section iteration procedure} we employ \eqref{ODI U, coupled with V} and \eqref{ODI V, coupled with U} as iteration scheme. However, in order to start with the iteration procedure we need to derive lower bound estimates for the integral nonlinear terms, so that, plugging these lower bounds in \eqref{ODI U, source |vt|^p} and \eqref{ODI V, source |u|^q} we get the first step of the iterative procedure. We will complete this task in the next section.

\section{Lower bounds for the spatial integral of the nonlinearities} \label{Section lower bounds nonlinear term}

As we have already announced the goal of this section is to determine lower bound estimates for the integrals of the semilinear terms. According to this purpose, we need to take into account the analysis of further auxiliary functionals related to the local solution $(u,v)$ of \eqref{weakly coupled system}. More specifically, we are going to estimate the functionals
\begin{align}
U_1(t)& \doteq \int_{\mathbb{R}^n}u(t,x)\Psi(t,x)\, dx,  \label{def U1} \\  
V_1(t)& \doteq \int_{\mathbb{R}^n}v(t,x)\Psi(t,x)\, dx , \label{def V1} \\
U_2(t)& \doteq \int_{\mathbb{R}^n} \partial_t u(t,x)\Psi(t,x)\, dx . \label{def V2} 
\end{align} In the definition of the functionals $U_1,V_1,U_2$ we used the function $\Psi=\Psi(t,x)\doteq e^{-t} \Phi(x)$,  where 
\begin{align}
\Phi=\Phi(x)\doteq \begin{cases} e^{x}+e^{-x} & \mbox{for} \ \ n=1, \\\displaystyle{\int_{\mathbb{S}^{n-1}} \, e^{\omega \cdot x}\, dS_{\omega}} & \mbox{for} \ \ n\geqslant 2\end{cases} \label{def eigenfunction laplace op}
\end{align} is an eigenfunction of the Laplace operator, as $\Delta \Phi =\Phi$. The test function $\Psi$ has been introduce for the first time in \cite{YZ06} in the study of the blow-up result for the semilinear classical wave equation with power nonlinearity in the critical case for high space dimension.

\begin{lem} \label{lemma U1,V1} Let $(w,\widetilde{w})$ be a local energy solution of the Cauchy problem
\begin{align*}
\begin{cases}
w_{tt}-\Delta w +b_1(t)w_t = G_1(t,x,w,w_t,\widetilde{w},\widetilde{w}_t),  & x\in \mathbb{R}^n, \ t\in (0,T),  \\
\widetilde{w}_{tt}-\Delta \widetilde{w} +b_2(t) \widetilde{w}_t = G_2(t,x,w,w_t,\widetilde{w},\widetilde{w}_t),  & x\in \mathbb{R}^n, \ t\in (0,T), \\
 (w,w_t,\widetilde{w},\widetilde{w}_t)(0,x)= (\varepsilon w_0, \varepsilon w_1, \varepsilon \widetilde{w}_0, \varepsilon \widetilde{w}_1)(x), & x\in \mathbb{R}^n,
\end{cases}
\end{align*} where the time-dependent coefficients of the damping terms $b_1,b_2\in \mathcal{C}([0,\infty))\cap L^1([0,\infty))$ and the nonlinear terms $G_1,G_2$ are nonnegative. Furthermore, we assume that $w_0,w_1,\widetilde{w}_0,\widetilde{w}_1$ are nonnegative, nontrivial and compactly supported and  that  $w, \widetilde{w}$ satisfy a support condition as in \eqref{support condition solution}.
 Let $W_1,\widetilde{W}_1$ be defined by $$ W_1(t) \doteq \int_{\mathbb{R}^n}w(t,x)\Psi(t,x)\, dx  \quad \mbox{and} \quad
\widetilde{W}_1(t) \doteq \int_{\mathbb{R}^n}\widetilde{w}(t,x)\Psi(t,x)\, dx  $$ for any $t\geqslant 0$. Then, for any $t\geqslant 0$ the following estimates hold
\begin{align*}
W_1(t) & \geqslant \varepsilon \,\tfrac{m_1(0)}{2}  \int_{\mathbb{R}^n}w_0(x) \Phi(x) \, dx \quad \mbox{and} \quad
\widetilde{W}_1(t)  \geqslant \varepsilon \,\tfrac{m_2(0)}{2} \int_{\mathbb{R}^n}\widetilde{w}_0(x) \Phi(x) \, dx. 
\end{align*}
\end{lem}

\begin{proof} See Lemma 2.2 in \cite{PalTak19dt}.
\end{proof}

In particular, from Lemma \ref{lemma U1,V1} we get immediately the lower bound estimates
\begin{align}
\label{lower bound U1}  U_1(t)&\geqslant \varepsilon I_1[u_0] \qquad \mbox{for any} \ t\geqslant 0 ,\\
\label{lower bound V1} V_1(t)&\geqslant \varepsilon I_2[v_0] \qquad \mbox{for any} \ t\geqslant 0 ,
\end{align} where $I_j[f]\doteq \frac{m_j(0)}{2}\int_{\mathbb{R}^n}f(x)\Phi(x)\, dx$ for $j=1,2$.

In the next step we follow the main ideas of \cite[Section 3]{LT18Glass} and \cite[Section 4]{LT18ComNon} in order to control the functional $U_2$ from below.

\begin{lem} Let $U_2$ be defined by \eqref{def V2}. Under the same assumptions of Theorem \ref{Thm blowup |vt|^p, |u|^q}, the following estimate holds
\begin{align}
\label{lower bound V2}
U_2(t)&\geqslant \varepsilon I_1[u_1] \qquad \mbox{for any} \ t\geqslant 0.
\end{align}
\end{lem}

\begin{proof}
Let us begin  pointing out that 
\begin{align}
\frac{d}{dt}& \bigg( m_1(t) \int_{\mathbb{R}^n}\big(\partial_t u(t,x)+u(t,x)\big) \Psi(t,x) \, dx \bigg) \notag\\ &= b_1(t) m_1(t) \int_{\mathbb{R}^n}\big(\partial_t u(t,x)+u(t,x)\big) \Psi(t,x) \, dx +m_1(t) \frac{d}{dt}  \int_{\mathbb{R}^n}\big(\partial_t u(t,x)+u(t,x)\big) \Psi(t,x) \, dx. \label{inter 1}
\end{align} 
Choosing $\psi\equiv \Psi$ in \eqref{def v}, we have
\begin{align*}
\int_{\mathbb{R}^n} & \partial_t u(t,x)\Psi(t,x)\,dx-\int_{\mathbb{R}^n}\varepsilon u_1(x)\Phi(x)\,dx - \int_0^t\int_{\mathbb{R}^n} \partial_t u(s,x)\Psi_s(s,x) \,dx\, ds\notag \\
&  +\int_0^t\int_{\mathbb{R}^n}\nabla u(s,x)\cdot\nabla\Psi(s,x)\, dx\, ds+\int_0^t\int_{\mathbb{R}^n}b_1(s)\partial_t u(s,x) \Psi(s,x)\,dx\, ds  =\int_0^t \int_{\mathbb{R}^n}|v(s,x)|^q\Psi(s,x)\,dx \, ds.
\end{align*} Differentiating both sides of the previous equality with respect to $t$, we arrive at
\begin{align}
\int_{\mathbb{R}^n}|v(t,x)|^q\Psi(t,x)\,dx & = \frac{d}{dt} \int_{\mathbb{R}^n} \partial_t u(t,x)\Psi(t,x)\, dx \notag \\ & \quad + \int_{\mathbb{R}^n} \big( - \partial_t u(t,x)\Psi_t(t,x)+\nabla u(t,x)\cdot\nabla\Psi(t,x) + b_1(t)\partial_t u(t,x) \Psi(t,x)\big)  \,dx . \label{inter a} 
\end{align} Using $\Delta \Psi=\Psi$ and $\Psi_t=-\Psi$, \eqref{inter a} yields
\begin{align}
\int_{\mathbb{R}^n}|v(t,x)|^q\Psi(t,x)\,dx &  = \frac{d}{dt} \int_{\mathbb{R}^n} \big( \partial_t u(t,x)+u(t,x)\big)\Psi(t,x)\, dx +b_1(t) \int_{\mathbb{R}^n} \partial_t u(t,x) \Psi(t,x) \,dx. \label{inter 3}
\end{align} If we combine \eqref{inter 1} and \eqref{inter 3}, we obtain 
\begin{align}
\frac{d}{dt} \bigg(  m_1(t) \int_{\mathbb{R}^n}\big(\partial_t u(t,x)& +u(t,x)\big) \Psi(t,x) \, dx \bigg)= b_1(t) m_1(t)\, U_1(t)  + m_1(t) \int_{\mathbb{R}^n}|v(t,x)|^q\Psi(t,x)\,dx , 
\label{inter 5}
\end{align} where $U_1$ is defined by \eqref{def U1}. 

\noindent
Thanks to \eqref{lower bound U1} we have that $U_1$ is nonnegative. Then, integrating \eqref{inter 5} over $[0,t]$, we get the estimate
\begin{align}
m_1(t) \int_{\mathbb{R}^n}\big(\partial_t u(t,x) &  +u(t,x)\big) \Psi(t,x) \, dx  \notag\\ & \geqslant \varepsilon \, m_1(t) \int_{\mathbb{R}^n}\big(u_0(x) +u_1(x)\big) \Phi(x) \, dx  + \int_0^t m_1(s) \int_{\mathbb{R}^n}|v(s,x)|^q\Psi(s,x)\,dx. \label{inter 7}
\end{align} 
Furthermore, we may rewrite \eqref{inter a} as follows
\begin{align}
\int_{\mathbb{R}^n}|v(t,x)|^q\Psi(t,x)\,dx & = \frac{d}{dt} \int_{\mathbb{R}^n} \partial_t u(t,x)\Psi(t,x)\, dx   + b_1(t) \int_{\mathbb{R}^n} \partial_t u(t,x) \Psi(t,x)  \,dx \notag \\ & \quad + \int_{\mathbb{R}^n} \big( \partial_t u(t,x)- u(t,x)\big)\Psi(t,x)\, dx.
\label{inter c}
\end{align} If we multiply both sides of \eqref{inter c} by $m_1(t)$, we find
\begin{align}
 \frac{d}{dt} \bigg(m_1(t)\int_{\mathbb{R}^n} \partial_t u(t,x) \Psi(t,x) \bigg)  +m_1(t) \int_{\mathbb{R}^n} &\big( \partial_t u(t,x)- u(t,x)\big)\Psi(t,x)\, dx \notag \\ &=  m_1(t) \int_{\mathbb{R}^n}|v(t,x)|^q\Psi(t,x)\,dx.
\label{inter 9}
\end{align}
Adding \eqref{inter 7} and \eqref{inter 9}, we find
\begin{align}
 \frac{d}{dt} \bigg(m_1(t)\int_{\mathbb{R}^n}  \partial_t u(t,x) & \Psi(t,x) \, dx\bigg)   + 2m_1(t) \int_{\mathbb{R}^n}  \partial_t u(t,x)\Psi(t,x)\, dx  \notag \\ & \geqslant  \varepsilon\, m_1(0) \int_{\mathbb{R}^n}\big(u_0(x) +u_1(x)\big) \Phi(x) \, dx + m_1(t) \int_{\mathbb{R}^n}|v(t,x)|^q\Psi(t,x)\,dx  \notag \\ & \quad + \int_0^t m_1(s) \int_{\mathbb{R}^n}|v(s,x)|^q\Psi(s,x)\,dx. \label{inter 12}
\end{align}
Let us set the auxiliary functional
\begin{align*}
U_3(t) & \doteq m_1(t) \int_{\mathbb{R}^n}  \partial_t u(t,x)  \Psi(t,x) \, dx -\varepsilon\, \frac{m_1(0)}{2} \int_{\mathbb{R}^n}u_1(x) \Phi(x) \, dx -\frac{1}{2} \int_0^t m_1(s) \int_{\mathbb{R}^n}|v(s,x)|^q\Psi(t,x)\,dx \, ds \, .
\end{align*}
Clearly, $U_3(0)=\varepsilon I_1[u_1]$.  Besides, \eqref{inter 12} implies
\begin{align}
U_3'(t)+2 U_3(t) & \geqslant \varepsilon\, m_1(0) \int_{\mathbb{R}^n}u_0(x) \Phi(x) \, dx + \frac{1}{2}\, m_1(t) \int_{\mathbb{R}^n}  |v(t,x)|^q  \Psi(t,x) \, dx \geqslant 0.\label{inter 13}
\end{align} Hence, multiplying \eqref{inter 13} by $e^{2t}$ and integrating over $[0,t]$, we get $U_3(t)\geqslant e^{-2t} U_3(0)\geqslant 0$. 
Therefore, as $U_3$ is nonnegative we may write
\begin{align*} 
   m_1(t) \int_{\mathbb{R}^n}  \partial_t u(t,x)  \Psi(t,x) \, dx & \geqslant \varepsilon\, \frac{m_1(0)}{2} \int_{\mathbb{R}^n}u_1(x) \Phi(x) \, dx +\frac{1}{2} \int_0^t m_1(s) \int_{\mathbb{R}^n}|v(s,x)|^q\Psi(t,x)\,dx \, ds \\ 
   & \geqslant \varepsilon\, \frac{m_1(0)}{2} \int_{\mathbb{R}^n}u_1(x) \Phi(x) \, dx
\end{align*} which implies immediately \eqref{lower bound V2} due to \eqref{boundedness multipliers}.
\end{proof}

Using \eqref{lower bound U1} and \eqref{lower bound V2}, we may finally derive the lower bounds for the integrals with respect to the spatial variables of the semilinear terms.

\begin{prop}\label{prop lower bound nonlinearity} Let $(u,v)$ be an energy solution of \eqref{weakly coupled system} on $[0,T)$ with nonnegative, continuous and summable coefficients of the damping terms $b_1,b_2$. Furthermore, we require the same assumptions on $u_0,u_1,v_0,v_1$ as in Theorem \ref{Thm blowup |vt|^p, |u|^q}. Then, the following estimates hold
\begin{align}
\int_{\mathbb{R}^n}|v(t,x)|^q \, dx &\geqslant \widetilde{C} \varepsilon^q\, (1+t)^{n-1-\frac{n-1}{2}q}, \label{lower bound |u|^q}\\
\int_{\mathbb{R}^n}|\partial_t u(t,x)|^p \, dx &\geqslant \widetilde{K}\varepsilon^p \, (1+t)^{n-1-\frac{n-1}{2}p} \label{lower bound |vt|^p}
\end{align} for any $t\geqslant 0$, where $\widetilde{C},\widetilde{K}$ are positive constants depending on $n,p,q,b_1,b_2,R,u_1,v_0$.
\end{prop}

\begin{rem} Let us underline explicitly that the conditions $u_1\not \equiv 0$ and $v_0\not \equiv 0$ guarantee that the multiplicative constants in \eqref{lower bound V1} and \eqref{lower bound V2} are positive. This fact will play a fundamental role in the proof of Proposition \ref{prop lower bound nonlinearity}.
\end{rem}

\begin{proof}
Let us prove \eqref{lower bound |u|^q}. By H\"older's inequality and $\supp v(t,\cdot)\subset B_{R+t}$ it follows
\begin{align*}
 \int_{\mathbb{R}^n}|v(t,x)|^q\, dx  &\geqslant \big(V_1(t)\big)^q\bigg(\int_{B_{R+t}}\big(\Psi(t,x)\big)^{q'}\, dx\bigg)^{-(q-1)} \\ &\gtrsim \Big(\varepsilon I_2[v_0]\Big)^q (1+t)^{n-1-\frac{n-1}{2}q},
\end{align*} where in the second inequality we used \eqref{lower bound V1} and the following estimate (cf. \cite[estimate (2.5)]{YZ06}):
\begin{align*}
\int_{B_{R+t}}\big(\Psi(t,x)\big)^{q'}\, dx \lesssim (1+t)^{n-1-\frac{n-1}{2}q'}.
\end{align*} Using \eqref{lower bound V2}, we can prove \eqref{lower bound |vt|^p} in a completely analogous way. 
\end{proof}

\section{Iteration argument} \label{Section iteration procedure}

In this section we combine the results from Sections \ref{Section iteration frame}-\ref{Section lower bounds nonlinear term} by using an iteration procedure to get a sequence of lower bound estimates for the functionals $V$ and $U'$ (for the definition of $U$ and $V$ see \eqref{def U,V} in Section \ref{Section iteration frame}).

More precisely, we want to prove that
\begin{align}
V(t) &\geqslant C_j (1+t)^{-b_j}t^{a_j} \qquad \mbox{ for any} \ t\geqslant 0 \label{U lower bound j}, \\
U'(t) &\geqslant K_j (1+t)^{-\beta_j}t^{\alpha_j} \qquad \mbox{for any} \ t\geqslant 0 \label{V' lower bound j},
\end{align} where $\{C_j\}_{j\in\mathbb{N}}$, $\{a_j\}_{j\in\mathbb{N}}$, $\{b_j\}_{j\in\mathbb{N}}$, $\{K_j\}_{j\in\mathbb{N}}$, $\{\alpha_j\}_{j\in\mathbb{N}}$ and $\{\beta_j\}_{j\in\mathbb{N}}$ are suitable sequences of nonnegative numbers that we will determine afterwards.

Our strategy is to prove \eqref{U lower bound j} and \eqref{V' lower bound j} by induction.

 Let us begin with the base case $j=0$. Plugging \eqref{lower bound |u|^q} in \eqref{ODI V, source |u|^q}, it results
\begin{align*}
U'(t) & \geqslant m_1(0)\widetilde{C} \varepsilon^q \int_0^t (1+s)^{n-1-\frac{n-1}{2}q} \, ds \geqslant  \tfrac{m_1(0)\widetilde{C}}{n}\,  \varepsilon^q  (1+t)^{-\frac{n-1}{2}q} t^{n}
\end{align*} which is \eqref{V' lower bound j} for $j=0$ provided that $K_0\doteq\frac{m_1(0)\widetilde{C}}{n}  \varepsilon^q $, $\alpha_0\doteq n$, $\beta_0\doteq \frac{n-1}{2}q$.
Analogously, combining \eqref{lower bound |vt|^p} and \eqref{ODI U, source |vt|^p}, we find
\begin{align*}
V(t)\geqslant  m_2(0)\widetilde{K} \varepsilon^p \int_0^t \int_0^s (1+\tau)^{n-1-\frac{n-1}{2}p} \, d\tau\, ds \geqslant 
\tfrac{m_2(0)\widetilde{K}}{n(n+1)} \, \varepsilon^p  (1+t)^{-\frac{n-1}{2}p} t^{n+1}.
\end{align*} So, we proved also \eqref{U lower bound j} for $j=0$ provided that $C_0\doteq\frac{m_2(0)\widetilde{K}}{n(n+1)}  \varepsilon^p $, $a_0\doteq n+1$, $b_0\doteq \frac{n-1}{2}p$.

Let us proceed now with the inductive step. If we plug \eqref{U lower bound j} in \eqref{ODI V, coupled with U}, then, for any $t\geqslant 0$ we have
\begin{align*}
U'(t) & \geqslant K C_j^q \int_0^t (1+s)^{-n(q-1)-b_jq} s^{a_jq} \, ds \geqslant K C_j^q  (1+t)^{-n(q-1)-b_jq}\int_0^t s^{a_jq} \, ds \\
& = K C_j^q (a_jq+1)^{-1} (1+t)^{-n(q-1)-b_jq} t^{a_jq+1}.
\end{align*} Thus, using the last lower bound in \eqref{ODI U, coupled with V}, we obtain for $t\geqslant 0$
\begin{align*}
V(t) & \geqslant C K^p C_j^{pq} (a_jq+1)^{-p} \int_0^t \int_0^s (1+\tau)^{-n(pq-1)-b_jpq}  \tau^{a_jpq+p}\, d\tau \, ds \\
& \geqslant C K^p C_j^{pq} (a_jq+1)^{-p}  (1+t)^{-n(pq-1)-b_jpq}\int_0^t \int_0^s   \tau^{a_jpq+p}\, d\tau \, ds \\
& = C K^p C_j^{pq} (a_jq+1)^{-p} (a_jpq+p+1)^{-1} (a_jpq+p+2)^{-1} (1+t)^{-n(pq-1)-b_jpq}   t^{a_jpq+p+2}.
\end{align*} Also, we proved \eqref{U lower bound j} for $j+1$ provided that $C_{j+1}\doteq C K^p C_j^{pq} (a_jq+1)^{-p} (a_jpq+p+1)^{-1} (a_jpq+p+2)^{-1}$, $a_{j+1}\doteq pq a_j+ p+2$ and $b_{j+1}\doteq b_j+n(pq-1)$.

\noindent
Similarly, if we plug \eqref{V' lower bound j} in \eqref{ODI U, coupled with V}, then, for any $t\geqslant 0$ we get
\begin{align*}
V(t) & \geqslant C K_j^p \int_0^t \int_0^s (1+\tau)^{-n(p-1)-\beta_j p}  \tau^{\alpha_j p}  \, d\tau\, ds \\
& \geqslant C K_j^p   (1+t)^{-n(p-1)-\beta_j p} \int_0^t \int_0^s \tau^{\alpha_j p}  \, d\tau\, ds \\
& = C K_j^p   (\alpha_j p+1)^{-1}(\alpha_j p+2)^{-1} (1+t)^{-n(p-1)-\beta_j p}  t^{\alpha_j p+2} .
\end{align*} Consequently, a combination of the last lower bound with \eqref{ODI V, coupled with U} yields 
\begin{align*}
U'(t) & \geqslant K C^q K_j^{pq} (\alpha_j p+1)^{-q}(\alpha_j p+2)^{-q} \int_0^t (1+s)^{-n(qp-1)-\beta_j pq}  s^{\alpha_j pq+2q}\, ds \\
& \geqslant K C^q K_j^{pq} (\alpha_j p+1)^{-q}(\alpha_j p+2)^{-q}(\alpha_j pq+2q+1)^{-1} (1+t)^{-n(qp-1)-\beta_j pq} t^{\alpha_j pq+2q+1}\, ds \\
\end{align*} for any $t\geqslant 0$. Hence, we proved \eqref{V' lower bound j} for $j+1$ provided that $\alpha_{j+1}\doteq pq \alpha_j+ 2q+1$, $\beta_{j+1}\doteq \beta_j+n(pq-1)$ and $K_{j+1}\doteq K C^q  K_j^{pq} (\alpha_j p+1)^{-q}(\alpha_j p+2)^{-q}(\alpha_j pq+2q+1)^{-1}$.

It is clear, from the recursive relations and from the nonnegative values of the initial constants $C_0$, $K_0$, $a_0$, $b_0$, $\alpha_0$, $\beta_0$, that  $C_j,K_j,a_j,b_j,\alpha_j,\beta_j$ are nonnegative real numbers for all $j\in\mathbb{N}$. Next we determine the explicit expressions for $a_j,b_j,\alpha_j,\beta_j$ and lower bound estimates for $C_j,K_j$. As $a_j=pq a_{j-1}+p+2$, employing iteratively this condition and the value $a_0=n+1$, we find
\begin{align*}
a_j = pq a_{j-1} +p+2 = \cdots = a_0 (pq)^j + (p+2) \sum_{k=0}^{j-1}(pq)^k =\Big(n+1+\tfrac{p+2}{pq-1}\Big)(pq)^j - \tfrac{p+2}{pq-1}.
\end{align*} Analogously,
\begin{align*}
\alpha_j &= \alpha_0 (pq)^j + (2q+1) \sum_{k=0}^{j-1}(pq)^k  = \Big(n+\tfrac{2q+1}{pq-1}\Big)(pq)^j - \tfrac{2q+1}{pq-1}, \\
b_j &= b_0 (pq)^j + n(pq-1) \sum_{k=0}^{j-1}(pq)^k  = \Big(\tfrac{n-1}{2}p+n \Big)(pq)^j - n, \\
\beta_j &= \beta_0 (pq)^j + n(pq-1) \sum_{k=0}^{j-1}(pq)^k  = \Big(\tfrac{n-1}{2}q+n \Big)(pq)^j - n.
\end{align*} In particular, using the representation formulas for $a_j$ and $\alpha_j$, we may derive lower bounds for $C_j$ and $K_j$. Indeed, due to 
\begin{align*}
a_{j-1}pq+p+2 & =  a_{j} \leqslant \Big(n+1+\tfrac{p+2}{pq-1}\Big)(pq)^j  , \\
\alpha_{j-1} pq+2q+1 & = \alpha_j  \leqslant \Big(n+\tfrac{2q+1}{pq-1}\Big)(pq)^j  ,
\end{align*}  we have
\begin{align}
C_{j} &= C K^p C_{j-1}^{pq} (a_{j-1}q+1)^{-p} (a_{j-1}qp+p+1)^{-1} (a_{j-1}pq+p+2)^{-1} \notag \\ &\geqslant C K^p C_{j-1}^{pq}  (a_{j-1}pq+p+2)^{-(p+2)}  \geqslant M  (pq)^{-(p+2)j} C_{j-1}^{pq} \label{Cj lower bound}
\end{align} and 
\begin{align}
K_j & = K C^q  K_{j-1}^{pq} (\alpha_j p+1)^{-q}(\alpha_j p+2)^{-q}(\alpha_j pq+2q+1)^{-1} \notag \\ & \geqslant K C^q  K_{j-1}^{pq} (\alpha_j pq+2q+1)^{-(2q+1)} \geqslant \widetilde{M} (pq)^{-(2q+1)j} K_{j-1}^{pq}, \label{Kj lower bound}
\end{align} where $M\doteq  C K^p \big(n+1+\tfrac{p+2}{pq-1}\big)^{-(p+2)}$ and $\widetilde{M}\doteq  K C^q \big(n+\tfrac{2q+1}{pq-1}\big)^{-(2q+1)}$.

 Applying the  logarithmic function to both sides of \eqref{Cj lower bound} and using in an iterative way the resulting estimate, we arrive at
\begin{align}
\log C_j  & \geqslant pq \log C_{j-1}  -j\log ((pq)^{p+2})+\log M \notag \\
& \geqslant (pq)^2 \log C_{j-2}  -(j+(j-1)pq)\log ((pq)^{p+2})+(1+pq)\log M \notag \\
& \geqslant \cdots \geqslant (pq)^j \log C_{0}  -\sum_{k=0}^{j-1} (j-k)(pq)^k\log ((pq)^{p+2})+\sum_{k=0}^{j-1}(pq)^k \log M \notag \\
&= (pq)^j \Big(\log C_0-\tfrac{pq}{(pq-1)^2}\log((pq)^{p+2})+\tfrac{\log M}{pq-1}\Big)  +(j+1)\,\tfrac{\log((pq)^{p+2})}{pq-1}+\tfrac{\log((pq)^{p+2})}{(pq-1)^2}-\tfrac{\log M}{pq-1}, \label{lower bound log Cj}
\end{align} where we used the formulas 
\begin{align}\label{sum formulas}
\sum_{k=0}^{j-1}(pq)^k=\frac{(pq)^{j}-1}{pq-1}, \qquad\sum_{k=0}^{j-1} (j-k)(pq)^k= \frac{1}{pq-1} \bigg(\frac{(pq)^{j+1}-1}{pq-1}-(j+1)\bigg),
\end{align} that can be proved via an inductive argument.

Therefore, for $j\geqslant j_1\doteq \lceil \frac{\log M}{\log ((pq)^{p+2})}-1-\frac{1}{pq-1} \rceil$ by \eqref{lower bound log Cj} we get 
\begin{align}
\log C_j  & \geqslant  (pq)^j \Big(\log C_0-\tfrac{pq}{(pq-1)^2}\log((pq)^{p+2})+\tfrac{\log M}{pq-1}\Big)= (pq)^j \log (N \varepsilon^p) ,\label{lower bound log Cj 2}
\end{align} where $N\doteq \frac{m_2(0)\widetilde{K}}{n(n+1)} ((pq)^{p+2})^{-\frac{pq}{(pq-1)^2}}M^{\frac{1}{pq-1}}$. 

Analogously, from \eqref{Kj lower bound} we derive the estimate
\begin{align}
\log K_j  & \geqslant  (pq)^j \Big(\log K_0-\tfrac{pq}{(pq-1)^2}\log((pq)^{2q+1})+\tfrac{\log \widetilde{M}}{pq-1}\Big)= (pq)^j \log (\widetilde{N} \varepsilon^q) \label{lower bound log Kj 2}
\end{align} for $j\geqslant j_2\doteq \lceil \frac{\log \widetilde{M}}{\log ((pq)^{2q+1})}-1-\frac{1}{pq-1} \rceil$, where $\widetilde{N}\doteq \frac{m_1(0)\widetilde{C}}{n} ((pq)^{2q+1})^{-\frac{pq}{(pq-1)^2}}\widetilde{M}^{\frac{1}{pq-1}}$. 

In the next section we will combine \eqref{U lower bound j}, \eqref{lower bound log Cj 2} and \eqref{V' lower bound j}, \eqref{lower bound log Kj 2} to complete the proof of Theorem  \ref{Thm blowup |vt|^p, |u|^q} in the case $\Theta_1(n,p,q)>0$ and in the case  $\Theta_2(n,p,q)>0$, respectively.

\section{conclusion of the proof of Theorem \ref{Thm blowup |vt|^p, |u|^q}} \label{Section end of the proof of the Theorem subcritical case}

Let us start with the case $\Theta_1(n,p,q)>0$.
Combining \eqref{U lower bound j} and \eqref{lower bound log Cj 2}, we have for $t\geqslant 0$ and $j\geqslant j_1$
\begin{align*}
V(t) & \geqslant \exp\big((pq)^j \log(N\varepsilon^p)\big) (1+t)^{-b_j} t^{a_j}   \\ & = \exp\Big((pq)^j \log\Big(N\varepsilon^p (1+t)^{-(\frac{n-1}{2}p+n)}t^{n+1+\frac{p+2}{pq-1}}\Big)\Big) (1+t)^{n} t^{-\frac{p+2}{pq-1}} . 
\end{align*} As for $t\geqslant 1$ it holds $(1+t)\leqslant 2t$, the previous estimate yields 
\begin{align}
V(t) & \geqslant \exp\Big((pq)^j \log\Big(2^{-(\frac{n-1}{2}p+n)}N\varepsilon^p t^{\frac{pq+p+1}{pq-1}-\frac{n-1}{2}p}\Big)\Big) (1+t)^{n} t^{-\frac{p+2}{pq-1}}   \notag \\ & = \exp\Big((pq)^j \log\Big(2^{-(\frac{n-1}{2}p+n)}N\varepsilon^p t^{p\Theta_1(n,p,q)}\Big)\Big) (1+t)^{n} t^{-\frac{p+2}{pq-1}} \notag\\
& = \exp\big((pq)^j \log\big(\varepsilon^p J(t)\big)\big) (1+t)^{n} t^{-\frac{p+2}{pq-1}}  \label{U lower bound J funct}
\end{align}
for $t\geqslant 1$, where $J(t)\doteq2^{-(\frac{n-1}{2}p+n)}N t^{p\Theta_1(n,p,q)} $. Consequently, we may choose $\varepsilon_0$ sufficiently small such that 
\begin{align*}
2^{(\frac{n-1}{2}+\frac{n}{p})\Theta_1(n,p,q)^{-1}}N^{-(p\Theta_1(n,p,q))^{-1}}\varepsilon_0^{\Theta_1(n,p,q)^{-1}}\geqslant 1.
\end{align*} So, for $\varepsilon\in (0,\varepsilon_0]$ and for $t\geqslant 2^{(\frac{n-1}{2}+\frac{n}{p})\Theta_1(n,p,q)^{-1}}N^{-(p\Theta_1(n,p,q))^{-1}}\varepsilon^{\Theta_1(n,p,q)^{-1}}$ it holds $J(t)>0$. Consequently, letting $j\to \infty$ in \eqref{U lower bound J funct}, the lower bound of $V(t)$ blows up and, then, $V(t)$ cannot be finite. Also, we proved that $V$ may be definite only for $t\lesssim \varepsilon^{\Theta_1(n,p,q)^{-1}}$.

Now, we prove the result in the case $\Theta_2(n,p,q)>0$.
Combining \eqref{V' lower bound j} and \eqref{lower bound log Kj 2}, we have for $t\geqslant 0$ and $j\geqslant j_2$
\begin{align*}
U'(t) & \geqslant \exp\big((pq)^j \log(\widetilde{N}\varepsilon^q)\big) (1+t)^{-\beta_j} t^{\alpha_j}   \\ & = \exp\Big((pq)^j \log\Big(\widetilde{N}\varepsilon^q(1+t)^{-(\frac{n-1}{2}q+n)}t^{n+\frac{2q+1}{pq-1}}\Big)\Big) (1+t)^{n} t^{-\frac{2q+1}{pq-1}} .
\end{align*} Then, for $t\geqslant 1$ it holds
\begin{align}
U'(t) & \geqslant \exp\Big((pq)^j \log\Big(2^{-(\frac{n-1}{2}q+n)}\widetilde{N}\varepsilon^q t^{q\Theta_2(n,p,q)}\Big)\Big) (1+t)^{n} t^{-\frac{2q+1}{pq-1}} \notag\\
& = \exp\big((pq)^j \log\big(\varepsilon^q \widetilde{J}(t)\big)\big) (1+t)^{n} t^{-\frac{p+2}{pq-1}}  \label{V' lower bound J funct}
\end{align}
for $t\geqslant 1$, where $\widetilde{J}(t)\doteq2^{-(\frac{n-1}{2}q+n)}\widetilde{N} t^{q\Theta_2(n,p,q)} $. Hence,  we can take $\varepsilon_0$ so small that 
\begin{align*}
2^{(\frac{n-1}{2}+\frac{n}{q})\Theta_2(n,p,q)^{-1}}\widetilde{N}^{-(q\Theta_2(n,p,q))^{-1}}\varepsilon_0^{\Theta_2(n,p,q)^{-1}}\geqslant 1.
\end{align*} Thus, for $\varepsilon\in (0,\varepsilon_0]$ and for $t\geqslant 2^{(\frac{n-1}{2}+\frac{n}{q})\Theta_2(n,p,q)^{-1}}\widetilde{N}^{-(q\Theta_2(n,p,q))^{-1}}\varepsilon^{\Theta_2(n,p,q)^{-1}}$ it holds $\widetilde{J}(t)>0$. Also, as $j\to \infty$ in \eqref{V' lower bound J funct}, the lower bound of $U'(t)$ diverges and $U'(t)$ is not finite. In this second case, we proved that $U'$ can be finite only for $t\lesssim \varepsilon^{\Theta_2(n,p,q)^{-1}}$. Combining the two possible cases, we proved the result and the upper bound estimate for the lifespan \eqref{lifespan upper bound estimate}.

\section{Critical case} \label{Section critical case}


In the critical case, we restrict our considerations to the not-damped case. 
Therefore, we shall consider the  weakly coupled system of semilinear wave equations \eqref{weakly coupled system crit} 
 in the critical case $\max\{\Theta_1(n,p,q),\Theta_2(n,p,q)\}=0$. We will generalize the approach from \cite{WakYor18,WakYor18damp} for a single semilinear equation and from \cite{PalTak19} for a weakly coupled system with power nonlinearities, in order to deal with the mixed type of nonlinear terms. 

For the sake of readability, we recall the definition of weak solution to \eqref{weakly coupled system crit}.

\begin{defn} \label{def weak sol crit case} Let $u_0,v_0\in H^1(\mathbb{R}^n)$ and $u_1,v_1\in L^2(\mathbb{R}^n)$.
We say that $(u,v)$ is a weak solution of \eqref{weakly coupled system crit} on $[0,T)$ if 
\begin{align*}
& u\in \mathcal{C}([0,T),H^1(\mathbb{R}^n))\cap \mathcal{C}^1([0,T),L^2(\mathbb{R}^n))\quad \mbox{and} \quad   \partial_ tu\in L^p_{\loc}([0,T)\times\mathbb{R}^n), \\
& v\in \mathcal{C}([0,T),H^1(\mathbb{R}^n))\cap \mathcal{C}^1([0,T),L^2(\mathbb{R}^n)) \quad \mbox{and} \quad   v\in  L^q_{\loc}([0,T)\times\mathbb{R}^n)
\end{align*}
satisfy 
the equalities
\begin{align} 
\int_{\mathbb{R}^n} & \Big( \partial_t u(t,x)\phi(t,x)-u(t,x)\phi_s(t,x)\Big)\,dx-\varepsilon\int_{\mathbb{R}^n}\Big( u_1(x)\phi(0,x)- u_0(x)\phi_s(0,x)\Big)\,dx \notag \\
&  + \int_0^t\int_{\mathbb{R}^n} u(s,x)\Big(\phi_{ss}(s,x)-\Delta \phi(s,x)\Big) \,dx\, ds   =\int_0^t \int_{\mathbb{R}^n}|v(s,x)|^q\phi(s,x)\,dx \, ds  \label{def u crit}
\end{align} and 
\begin{align} 
\int_{\mathbb{R}^n} & \Big( \partial_t v(t,x)\psi(t,x)-v(t,x)\psi_s(t,x)\Big)\,dx-\varepsilon\int_{\mathbb{R}^n}\Big( v_1(x)\psi(0,x)- v_0(x)\psi_s(0,x)\Big)\,dx \notag \\
&  + \int_0^t\int_{\mathbb{R}^n} v(s,x)\Big(\psi_{ss}(s,x)-\Delta \psi(s,x)\Big) \,dx\, ds   =\int_0^t \int_{\mathbb{R}^n}|\partial_t u(s,x)|^p\phi(s,x)\,dx \, ds \label{def v crit}
\end{align}
for any test functions $\phi,\psi \in \mathcal{C}_0^\infty([0,T)\times\mathbb{R}^n)$ and any $t\in [0,T)$.
\end{defn}

The remaining part of this section is organized as follows: first, in Section \ref{Subsection functionals crit case} we recall some auxiliary functions from \cite{WakYor18} and we use them to introduce the functionals for the critical case;
in Section \ref{Subsection iteration frame} we derive the iteration frame for these functionals, that is, a coupled system of nonlinear ordinary integral inequalities;
in Section \ref{Subsection lower bound log}  lower bound estimates for the functionals, that allow to start with the iteration procedure, are derived;
then, in Section \ref{Subsection iteration argument slicing method} we combine the iteration frame from Section \ref{Subsection iteration frame} and the lower bounds from Section \ref{Subsection lower bound log} with a slicing method;
hence, in Section \ref{Subsection lifespan critical case} we use the sequences of lower bounds for the functionals from Section \ref{Subsection iteration argument slicing method}  to prove the blow-up result and to establish the upper bound for the lifespan;
finally, in Section \ref{Subsection comparison with ISW} we compare our results with those proved in Section 9 of \cite{ISW18} and we provide the analytic expression of the coordinates of  the cusp point for the critical curve in the $p$ - $q$ plane.

\subsection{Introduction of the functionals for the critical case} 
\label{Subsection functionals crit case}

Throughout the treatment of the critical case we will employ the auxiliary functions 
\begin{align*}
\eta_r(t,s,x) & \doteq \int_0^{\lambda_0} e^{-\lambda(R+t)} \frac{\sinh \lambda (t-s)}{ \lambda (t-s)} \, \Phi(\lambda x) \, \lambda^r \, d\lambda, \\
\xi_r(t,s,x) & \doteq \int_0^{\lambda_0} e^{-\lambda(R+t)} \cosh \lambda (t-s) \, \Phi(\lambda x) \, \lambda^r \, d\lambda,
\end{align*} where $r>-1$, $\lambda_0$ is a fixed positive constant and $\Phi$ is defined by \eqref{def eigenfunction laplace op}. These auxiliary functions have been introduced in \cite{WakYor18} as generalizations of the test function considered by Zhou (see \cite[equation (3.2)]{Zhou07}) in the treatment of the critical case for the semilinear wave equation with power nonlinearity in the higher dimensional case. Let us underline that the assumption on $r$ is done in order to guarantee the integrability of the function $\lambda^r$ in a neighborhood of $0$.

As functionals to study the blow-up dynamic we will consider 
\begin{align}
\mathcal{U}(t) & \doteq \int_{\mathbb{R}^n} \partial_t u(t,x) \, \eta_{r_1}(t,t,x) \, dx , \label{def mathcalU}\\
\mathcal{V}(t) & \doteq \int_{\mathbb{R}^n} v(t,x)\, \eta_{r_2}(t,t,x) \, dx. \label{def mathcalV}
\end{align} We point out that the choice of the conditions for the pair $(r_1,r_2)$ depends on the critical case we deal with. More specifically, we have to distinguish among the three possible subcases $\Theta_1(n,p,q)=0>\Theta_2(n,p,q)$, $\Theta_1(n,p,q)<0=\Theta_2(n,p,q)$ and $\Theta_1(n,p,q)=\Theta_2(n,p,q)=0$.

First, we derive two fundamental identities for $\mathcal{U}$ and $\mathcal{V}$, which involve the initial data and the nonlinear terms.

\begin{prop} Let $(u,v)$ be a weak solution of \eqref{weakly coupled system crit} on $[0,T)$ and let $\mathcal{U,\mathcal{V}}$ denote the functionals defined by \eqref{def mathcalU}, \eqref{def mathcalV}. Then, the following identities are satisfied for any $t\geqslant 0$:
\begin{align}
\mathcal{U}(t) & = \varepsilon t \int_{\mathbb{R}^n} u_0(x)\,\eta_{r_1+2}(t,0,x) \, dx + \varepsilon  \int_{\mathbb{R}^n} u_1(x)\,\xi_{r_1}(t,0,x) \, dx +\int_0^t \int_{\mathbb{R}^n} |v(s,x)|^q\,\xi_{r_1}(t,s,x) \, dx \, ds,
\label{fundamental identity mathcalU} \\
\mathcal{V}(t) & = \varepsilon  \int_{\mathbb{R}^n} v_0(x)\,\xi_{r_2}(t,0,x) \, dx + \varepsilon t  \int_{\mathbb{R}^n} v_1(x)\,\eta_{r_2}(t,0,x) \, dx +\int_0^t (t-s) \int_{\mathbb{R}^n} |\partial_t u(s,x)|^p\,\eta_{r_2}(t,s,x) \, dx \, ds.
\label{fundamental identity mathcalV}
\end{align}
\end{prop}

\begin{proof}
In order to show the validity of \eqref{fundamental identity mathcalU} and \eqref{fundamental identity mathcalV} we will employ the definition of weak solution for \eqref{weakly coupled system crit} with a suitable choice of the test functions $(\phi,\psi)$ in \eqref{def u crit} and \eqref{def v crit}. 
If we assume that $(u,v)$ satisfies \eqref{support condition solution}, then, $\supp u(t,\cdot), \supp v(t,\cdot)\subset B_{R+t}$ for any $t\geqslant 0$. Therefore, we may remove the assumption of compactness for the supports of the test functions in Definition \ref{def weak sol crit case}. Hence, it is possible to consider 
\begin{align*}
\phi &=\phi(t;s,x)= \cosh \lambda (t-s) \, \Phi(\lambda x), \\
\psi &=\psi(t;s,x)=\frac{\sinh \lambda (t-s)}{\lambda} \,  \Phi(\lambda x).
\end{align*} Since $\Delta \Phi(\lambda x)= \lambda^2 \Phi(\lambda x)$, then, $\phi,\psi$ are solutions of the homogeneous free wave equation. Moreover, 
\[\begin{array}{llll}
 \phi(t;t,x)=\Phi(\lambda x),  & \phi(t;0,x)=\cosh \lambda t \, \Phi(\lambda x),   & \phi_s(t;t,x)=0, & \phi_s(t;0,x)=-\lambda \sinh \lambda t \, \Phi(\lambda x), \\
 \psi(t;t,x) =0,  & \psi(t;0,x)= \lambda^{-1} \sinh \lambda t \, \Phi(\lambda x),   & \psi_s(t;t,x)=\! -\Phi(\lambda x),  & \psi_s(t;0,x)=-\cosh \lambda x \, \Phi(\lambda x). 
\end{array}\]
Consequently, from \eqref{def u crit} and \eqref{def v crit} we obtain 
\begin{align}
\int_{\mathbb{R}^n}\partial_t u(t,x) \Phi(\lambda x) \, dx & =  \varepsilon \lambda \sinh \lambda t \int_{\mathbb{R}^n} u_0(x)\Phi(\lambda x) \, dx  + \varepsilon \cosh \lambda t \int_{\mathbb{R}^n} u_1(x)\Phi(\lambda x) \, dx  \notag \\ & \quad +\int_0^t \int_{\mathbb{R}^n} |v(s,x)|^q \cosh \lambda (t-s) \Phi(\lambda x) \, dx \, ds \label{u intermediate estimate} \\
\int_{\mathbb{R}^n} v(t,x) \Phi(\lambda x) \, dx & =  \varepsilon \cosh \lambda t \int_{\mathbb{R}^n} v_0(x)\Phi(\lambda x) \, dx + \varepsilon \lambda^{-1} \sinh \lambda t \int_{\mathbb{R}^n} v_1(x)\Phi(\lambda x) \, dx \notag \\ & \quad +\int_0^t \int_{\mathbb{R}^n} |\partial_t u(s,x)|^p \lambda^{-1} \sinh \lambda (t-s) \Phi(\lambda x) \, dx \, ds. \label{v intermediate estimate}
\end{align}
Multiplying both sides of \eqref{u intermediate estimate} by $e^{-\lambda(R+t)}\lambda^{r_1}$, integrating the resulting relation with respect to $\lambda$ over $[0,\lambda_0]$ and, finally, applying Fubini's theorem, we get \eqref{fundamental identity mathcalU}. Similarly, from \eqref{v intermediate estimate} we find \eqref{fundamental identity mathcalV}. This concludes the proof.
\end{proof}

The next step is to derive from \eqref{fundamental identity mathcalU} and \eqref{fundamental identity mathcalV} the iteration frame. In order to do so, we need to estimate sharply the auxiliary functions $\eta_r$ and $\xi_r$.
\begin{lem} \label{lemma eta and xi estimates}Let $n\geqslant 2$. There exist $\lambda_0>0$ such that the following properties hold:
\begin{itemize}
\item[\rm{(i)}] if $r>-1$, $|x|\leqslant R$ and $t\geqslant 0$, then, 
\begin{align*}
\xi_r(t,0,x) & \geqslant A_0, \\
\eta_r(t,0,x) & \geqslant B_0 \langle t\rangle^{-1};
\end{align*}
\item[\rm{(ii)}] if $r>-1$, $|x|\leqslant s+R$ and $t>s\geqslant 0$, then, 
\begin{align*}
\xi_r(t,s,x) & \geqslant A_1  \langle s\rangle^{-r-1},\\
\eta_r(t,s,x) & \geqslant B_1 \langle t\rangle^{-1} \langle s\rangle^{-r};
\end{align*}
\item[\rm{(iii)}] if $r>\frac{n-3}{2}$, $|x|\leqslant t+R$ and $t> 0$, then, 
\begin{align*}
\eta_r(t,t,x) & \leqslant B_2 \langle t\rangle^{-\frac{n-1}{2}} \langle t-|x| \rangle^{\frac{n-3}{2}-r}.
\end{align*}
\end{itemize}
Here $A_0$ and $B_k$, $k=0,1,2$, are positive constants depending only on $\lambda_0$, $r$ and $R$ and we denote $\langle y\rangle \doteq 3+|y|$.
\end{lem}

\begin{rem} Let us stress that differently from \cite[Lemma 3.1]{WakYor18} we require in the statement of (i) and (ii) the condition of $r>-1$ instead of $r>0$. Nonetheless, the proofs from \cite{WakYor18} of (i) and of the lower bound for $\eta_r(t,s,x)$ in (ii) are still valid even for $r>-1$. 
\end{rem}

\begin{proof} We can restrict our considerations to the lower bound estimate for $\xi(t,s,x)$ in (ii), as the other properties are already proved in \cite[Lemma 3.1]{WakYor18}. Since $\langle s \rangle \geqslant 2$, we may shrink the domain of integration in the definition of $\xi_r(t,s,x)$ as follows 
\begin{align*}
\xi_r(t,s,x) & \geqslant  \int_{\lambda_0/ \langle s \rangle}^{2\lambda_0/ \langle s \rangle} e^{-\lambda(R+t)} \cosh \lambda (t-s) \, \Phi(\lambda x) \, \lambda^r \, d\lambda.
\end{align*}
 We remark that the condition $$\Phi(x)\asymp \langle x\rangle^{-\frac{n-1}{2}}e^{|x|} \quad \mbox{for any} \ \ x\in \mathbb{R}^n$$ implies that the infimum $$  \inf_{\lambda\in \left[\frac{\lambda_0}{\langle s\rangle} , \frac{2\lambda_0}{\langle s\rangle}\right]}\inf_{|x|\leqslant s+R} e^{-\lambda (s+R)}\Phi(\lambda x) $$ can be estimate from below by a constant $A=A(\lambda_0,R)>0$ that does not depend on $\lambda$, $s$ and $x$. Therefore, we may estimate
 \begin{align*}
\xi_r(t,s,x) & \geqslant    \int_{\lambda_0/ \langle s \rangle}^{2\lambda_0/ \langle s \rangle}e^{-\lambda (t-s)} \cosh \lambda (t-s) \, e^{-\lambda(R+s)}  \,\Phi(\lambda x) \, \lambda^r \, d\lambda \\ & = \int_{\lambda_0/ \langle s \rangle}^{2\lambda_0/ \langle s \rangle}\frac {1}{2} \left(1+e^{-2\lambda (t-s)}\right) \, e^{-\lambda(R+s)}  \,\Phi(\lambda x) \, \lambda^r \, d\lambda \geqslant A \int_{\lambda_0/ \langle s \rangle}^{2\lambda_0/ \langle s \rangle}\frac {1}{2} \left(1+e^{-2\lambda (t-s)}\right) \, \lambda^r \, d\lambda \\ & \geqslant \frac {A}{2} \int_{\lambda_0/ \langle s \rangle}^{2\lambda_0/ \langle s \rangle}  \, \lambda^r \, d\lambda =\frac {A}{2} \frac{\lambda_0^{r+1}}{(r+1)}(2^{r+1}-1)\langle s \rangle^{-r-1},
\end{align*} which is the desired lower bound estimate for $\xi_r(t,s,x)$.
\end{proof}

\subsection{Derivation of the iteration frame in the critical case}
\label{Subsection iteration frame} 

In order to derive the iteration scheme, we have to consider separately the three critical cases. In each case we will fix suitable conditions on the pair $(r_1,r_2)$, that will influence, on the one hand, the structure of the scheme itself with the possible presence of a logarithmic factor in the integral inequalities and, on the other hand, the functional $\mathcal{U}$ and/or $\mathcal{V}$ for which we can derive a lower bound containing a logarithmic factor.

\subsubsection{Case $\Theta_1(n,p,q)=0$} \label{Subsubsection Theta1=0 OIIs}


In this case we consider $r_1=\frac{n-1}{2}-\frac{1}{p}$ and $r_2>\frac{n-1}{2}-\frac{1}{q}$. The purpose of this section is to derive the frame for the iteration argument, which is a coupled system of integral inequalities for the functionals $\mathcal{U},\mathcal{V}$. In order to get this system we will combine the fundamental identities \eqref{fundamental identity mathcalU}, \eqref{fundamental identity mathcalV} and the estimates for the auxiliary functions in Lemma \ref{lemma eta and xi estimates}. Combining \eqref{support condition solution}, \eqref{def mathcalU} and H\"older's inequality, we find 
\begin{align}
\mathcal{U}(s) \leqslant \bigg(\int_{\mathbb{R}^n}|\partial_t u(s,x)|^p \eta_{r_2}(t,s,x)\, dx\bigg)^{\frac{1}{p}}  \Bigg(\int_{B_{R+s}}\frac{\eta_{r_1}(s,s,x)^{p'}}{\eta_{r_2}(t,s,x)^{\frac{p'}{p}}}\,dx\Bigg)^{\frac{1}{p'}}. \label{mathcal U holder}
\end{align} 
Using Lemma \ref{lemma eta and xi estimates} (ii)-(iii) and the condition $r_1=\frac{n-1}{2} -\frac{1}{p}$, we may estimate
\begin{align*}
\int_{B_{R+s}}\frac{\eta_{r_1}(s,s,x)^{p'}}{\eta_{r_2}(t,s,x)^{\frac{p'}{p}}}\,dx & \lesssim \langle t \rangle^{\frac{p'}{p}}\langle s\rangle^{r_2 \frac{p'}{p}-\frac{n-1}{2}p'} \int_{B_{R+s}} \langle s-|x|\rangle^{\left(\frac{n-3}{2}-r_1\right)p'}\, dx  \\ & \lesssim \langle t \rangle^{\frac{p'}{p}}\langle s\rangle^{r_2 \frac{p'}{p}-\frac{n-1}{2}p'+n-1}\log \langle s\rangle  .
\end{align*} Therefore, we get
\begin{align*}
\int_{\mathbb{R}^n}|\partial_t u(s,x)|^p \eta_{r_2}(t,s,x)\, dx & \gtrsim  \big(\mathcal{U}(s)\big)^p \Bigg(\int_{B_{R+s}}\frac{\eta_{r_1}(s,s,x)^{p'}}{\eta_{r_2}(t,s,x)^{\frac{p'}{p}}}\,dx\Bigg)^{-\frac{p}{p'}} \\ & \gtrsim \langle t \rangle^{-1}\langle s\rangle^{-r_2 +\frac{n-1}{2}p-(n-1)(p-1)}(\log \langle s\rangle )^{-(p-1)} \big(\mathcal{U}(s)\big)^p.
\end{align*} Consequently, from \eqref{fundamental identity mathcalV} we obtain 
\begin{align}\label{OII mathcalV case Theta1}
\mathcal{V}(t) \gtrsim \langle t \rangle^{-1} \int_0^t (t-s) \langle s\rangle^{-r_2 +\frac{n-1}{2}p-(n-1)(p-1)}(\log \langle s\rangle )^{-(p-1)} \big(\mathcal{U}(s)\big)^p \, ds.
\end{align}
Now we will derive an analogous integral lower bound for $\mathcal{U}$. By \eqref{def mathcalV} and H\"older's inequality we have
\begin{align}
\mathcal{V}(s) \leqslant \bigg(\int_{\mathbb{R}^n}|v(s,x)|^q \xi_{r_1}(t,s,x)\, dx\bigg)^{\frac{1}{q}}  \Bigg(\int_{B_{R+s}}\frac{\eta_{r_2}(s,s,x)^{q'}}{\xi_{r_1}(t,s,x)^{\frac{q'}{q}}}\,dx\Bigg)^{\frac{1}{q'}}. \label{mathcal V holder}
\end{align} Employing again Lemma \ref{lemma eta and xi estimates} and the condition $r_2> \frac{n-1}{2} -\frac{1}{q}$, we arrive at
\begin{align*}
\int_{B_{R+s}}\frac{\eta_{r_2}(s,s,x)^{q'}}{\xi_{r_1}(t,s,x)^{\frac{q'}{q}}} \,dx & \lesssim \langle s\rangle^{(r_1+1) \frac{q'}{q}-\frac{n-1}{2}q'} \int_{B_{R+s}} \langle s-|x|\rangle^{\left(\frac{n-3}{2}-r_2\right)q'}\, dx  \\ & \lesssim \langle s\rangle^{(r_1+1) \frac{q'}{q}-\frac{n-1}{2}q'+n+\left(\frac{n-3}{2}-r_2\right)q' },
\end{align*} which implies in turn
\begin{align*}
 \int_{\mathbb{R}^n}|v(s,x)|^q \xi_{r_1}(t,s,x)\, dx \, ds & \gtrsim   \big(\mathcal{V}(s)\big)^q \Bigg(\int_{B_{R+s}}\frac{\eta_{r_2}(s,s,x)^{q'}}{\xi_{r_1}(t,s,x)^{\frac{q'}{q}}}\,dx\Bigg)^{-\frac{q}{q'}}  \\
 & \gtrsim  \langle s\rangle^{-(r_1+1) +\frac{n-1}{2}q-n(q-1)-\left(\frac{n-3}{2}-r_2\right)q } \big(\mathcal{V}(s)\big)^q .
\end{align*} Finally, \eqref{fundamental identity mathcalU} and the previous inequality yield
\begin{align} \label{OII mathcalU case Theta1}
\mathcal{U}(t) &\gtrsim \int_0^t  \langle s\rangle^{-r_1+n-1 -(n-1)q +r_2q } \big(\mathcal{V}(s)\big)^q\, ds.
\end{align}

\subsubsection{Case $\Theta_2(n,p,q)=0$} \label{Subsubsection Theta2=0 OIIs}

For this critical case we assume $r_1>\frac{n-1}{2}-\frac{1}{p}$ and $r_2=\frac{n-1}{2}-\frac{1}{q}$.
Due to the fact that we switch in some sense the role of $r_1$ and $r_2$ with respect to the previous critical case $\Theta_1(n,p,q)=0$, somehow also the structure of the iteration frame is reversed with respect to the previous section.

By Lemma \ref{lemma eta and xi estimates} (ii)-(iii) and the condition $r_1>\frac{n-1}{2} -\frac{1}{q}$ it follows
\begin{align*}
\int_{B_{R+s}}\frac{\eta_{r_1}(s,s,x)^{p'}}{\eta_{r_2}(t,s,x)^{\frac{p'}{p}}}\,dx & \lesssim \langle t \rangle^{\frac{p'}{p}}\langle s\rangle^{r_2 \frac{p'}{p}-\frac{n-1}{2}p'} \int_{B_{R+s}} \langle s-|x|\rangle^{\left(\frac{n-3}{2}-r_1\right)p'}\, dx  \\ & \lesssim \langle t \rangle^{\frac{p'}{p}}\langle s\rangle^{r_2 \frac{p'}{p}-\frac{n-1}{2}p'+n+\left(\frac{n-3}{2}-r_1\right)p'}.
\end{align*} Then, from \eqref{mathcal U holder} we get
\begin{align*}
\int_{\mathbb{R}^n}|\partial_t u(s,x)|^p \eta_{r_2}(t,s,x)\, dx 
&  \gtrsim \langle t \rangle^{-1}\langle s\rangle^{-r_2 +\frac{n-1}{2}p-n(p-1)-\left(\frac{n-3}{2}-r_1\right)p} \big(\mathcal{U}(s)\big)^p \\
&  \gtrsim \langle t \rangle^{-1}\langle s\rangle^{-r_2 -(n-1)p+n+r_1 p} \big(\mathcal{U}(s)\big)^p.
\end{align*} Also, \eqref{fundamental identity mathcalV} yields
\begin{align}\label{OII mathcalV case Theta2}
\mathcal{V}(t) \gtrsim \langle t \rangle^{-1} \int_0^t (t-s) \langle s\rangle^{-r_2 -(n-1)p+n+r_1 p}\big(\mathcal{U}(s)\big)^p\, ds.
\end{align}
We determine now the integral lower bound for $\mathcal{U}$. By using Lemma \ref{lemma eta and xi estimates} and the condition $r_2= \frac{n-1}{2} -\frac{1}{q}$, we arrive at
\begin{align*}
\int_{B_{R+s}}\frac{\eta_{r_2}(s,s,x)^{q'}}{\xi_{r_1}(t,s,x)^{\frac{q'}{q}}} \,dx & \lesssim \langle s\rangle^{(r_1+1) \frac{q'}{q}-\frac{n-1}{2}q'} \int_{B_{R+s}} \langle s-|x|\rangle^{\left(\frac{n-3}{2}-r_2\right)q'}\, dx  \\ & \lesssim \langle s\rangle^{(r_1+1) \frac{q'}{q}-\frac{n-1}{2}q'+n-1 }\log\langle s \rangle.
\end{align*} The last estimate together with \eqref{mathcal V holder} provides
\begin{align*}
 \int_{\mathbb{R}^n}|v(s,x)|^q \xi_{r_1}(t,s,x)\, dx  
 & \gtrsim  \langle s\rangle^{-(r_1+1) +\frac{n-1}{2}q-(n-1)(q-1)} (\log\langle s \rangle)^{-(q-1)} \big(\mathcal{V}(s)\big)^q.
\end{align*} Thus, \eqref{fundamental identity mathcalU} and the last estimate imply
\begin{align} \label{OII mathcalU case Theta2}
\mathcal{U}(t) &\gtrsim \int_0^t   \langle s\rangle^{-(r_1+1) +\frac{n-1}{2}q-(n-1)(q-1)} (\log\langle s \rangle)^{-(q-1)} \big(\mathcal{V}(s)\big)^q \, ds.
\end{align}

\subsubsection{Case $\Theta_1(n,p,q)=\Theta_2(n,p,q)=0$}

 In this case we choose $r_1=\frac{n-1}{2}-\frac{1}{p}$ and $r_2=\frac{n-1}{2}-\frac{1}{q}$.  In particular, one can prove the identities 
\begin{align}
\tfrac{n-1}{2}-\tfrac{1}{p} &=n-1-\tfrac{n-1}{2}q , \label{identity r1} \\
\tfrac{n-1}{2}-\tfrac{1}{q} &=n-\tfrac{n-1}{2}p, \label{identity r2}
 \end{align} 
 due to the fact that the pair $(p,q)$ satisfies both the critical conditions $\Theta_1(n,p,q)=\Theta_2(n,p,q)=0$. Indeed, if we denote $\kappa_1=n-1-\tfrac{n-1}{2}q-\tfrac{n-1}{2}+\tfrac{1}{p}$ and $\kappa_1=n-\tfrac{n-1}{2}p-\tfrac{n-1}{2}+\tfrac{1}{q}$, then
 \begin{align*}
 \kappa_1+q\kappa_2 &= (pq-1)\Theta_1(n,p,q)=0, \\
 p\kappa_1+\kappa_2 &= (pq-1)\Theta_2(n,p,q)=0.
 \end{align*} As $pq\neq 1$, then, trivially $\kappa_1=\kappa_2=0$, but this means exactly the validity of \eqref{identity r1}-\eqref{identity r2}.
 
Since $r_1=\frac{n-1}{2}-\frac{1}{p}$ as in Section \ref{Subsubsection Theta1=0 OIIs}, we can prove \eqref{OII mathcalV case Theta1}. However, thanks to \eqref{identity r2} we see that the power of $\langle s \rangle$ in the right hand side of \eqref{OII mathcalV case Theta1} is exactly $-1$, that is,
\begin{align}\label{OII mathcalV case Theta1Theta2}
\mathcal{V}(t) \gtrsim \langle t \rangle^{-1} \int_0^t (t-s) \langle s\rangle^{-1}(\log \langle s\rangle )^{-(p-1)} \big(\mathcal{U}(s)\big)^p \, ds.
\end{align}

Similarly, since $r_2=\frac{n-1}{2}-\frac{1}{q}$ as in Section \ref{Subsubsection Theta2=0 OIIs} it holds \eqref{OII mathcalU case Theta2}. Yet, due to \eqref{identity r1} we find again that  the power of $\langle s \rangle$ in the right hand side of \eqref{OII mathcalU case Theta2} is exactly $-1$, that is,
\begin{align}\label{OII mathcalU case Theta1Theta2}
\mathcal{U}(t) \gtrsim \int_0^t  \langle s\rangle^{-1}(\log \langle s\rangle )^{-(q-1)} \big(\mathcal{V}(s)\big)^q \, ds.
\end{align}

\subsection{Lower bound estimates for the functionals containing a logarithmic factor} \label{Subsection lower bound log}

Purpose of this section is to derive lower bounds for $\mathcal{U}$ and/or $\mathcal{V}$ of logarithmic type. 
As in the previous section, we shall consider separately the three critical cases. We point out that the assumptions on the pair $(r_1,r_2)$ are the same as in Section \ref{Subsection iteration frame} and they depend on the critical case that we consider.

\subsubsection{Case $\Theta_1(n,p,q)=0$ }

In this case we will derive a lower bound for the functional $\mathcal{U}$ in two step. From \eqref{fundamental identity mathcalV}, Lemma \ref{lemma eta and xi estimates} (ii) and Proposition \ref{prop lower bound nonlinearity}, we get for $t\geqslant 0$
\begin{align}
\mathcal{V}(t) & \geqslant \int_0^t (t-s) \int_{\mathbb{R}^n}|\partial_t u(s,x)|^p \eta_{r_2}(t,s,x) \, dx \, ds \gtrsim \langle t \rangle^{-1} \int_0^t (t-s)\langle s \rangle^{-r_2} \int_{\mathbb{R}^n}|\partial_t u(s,x)|^p  \, dx \, ds \notag \\
& \gtrsim \varepsilon ^p \langle t \rangle^{-1} \int_0^t (t-s)\langle s \rangle^{-r_2+n-1-\frac{n-1}{2}p}  \, ds. \label{lower bound mathcalV intermediate}
\end{align} Consequently, for $t\geqslant 1$
\begin{align*}
\mathcal{V}(t)  & \gtrsim \varepsilon ^p \langle t \rangle^{-1-r_2-\frac{n-1}{2}p} \int_0^t (t-s)\langle s \rangle^{n-1}  \, ds  \gtrsim \varepsilon ^p \langle t \rangle^{-1-r_2-\frac{n-1}{2}p} \int_{\tfrac{t}{2}}^t (t-s)\langle s \rangle^{n-1} \, ds \\
& \gtrsim \varepsilon ^p \langle t \rangle^{-1-r_2-\frac{n-1}{2}p} \langle \tfrac{t}{2} \rangle^{n-1} \int_{\tfrac{t}{2}}^t (t-s)  \, ds  \gtrsim\varepsilon ^p \langle t \rangle^{-r_2-\frac{n-1}{2}p+n} .
\end{align*} Plugging the last lower bound for $\mathcal{V}$ in \eqref{OII mathcalU case Theta1}, we have for $t\geqslant 1$
\begin{align}
\mathcal{U}(t) &\gtrsim  \varepsilon^{pq}  \int_1^t  \langle s\rangle^{-r_1+n-1 -(n-1)q +r_2q +\left(-r_2-\frac{n-1}{2}p+n\right)q} \, ds \notag \\
&\gtrsim  \varepsilon^{pq}  \int_1^t  \langle s\rangle^{-r_1+n-1 +q -\frac{n-1}{2}pq} \, ds  \gtrsim  \varepsilon^{pq}  \int_1^t  \langle s\rangle^{ q+p^{-1} -\frac{n-1}{2}(pq-1)} \, ds \notag \\ &\gtrsim \varepsilon^{pq}  \int_1^t  \langle s\rangle^{ -1} ds \gtrsim \varepsilon^{pq}  \int_1^t   s^{-1}  ds \gtrsim \varepsilon^{pq} \log t , \label{lower bound mathcalU log case Theta1}
\end{align} where we used in the third inequality the actual value of $r_1$ and in the fourth one the critical condition  $\Theta_1(n,p,q)=0$.

\subsubsection{Case $\Theta_2(n,p,q)=0$ }

Let us determine a lower bound for $\mathcal{V}$ in two step. From \eqref{fundamental identity mathcalU}, Lemma \ref{lemma eta and xi estimates} (ii) and Proposition \ref{prop lower bound nonlinearity} we obtain for $t\geqslant 0$
\begin{align}
\mathcal{U}(t) & \geqslant \int_0^t \int_{\mathbb{R}^n}|v(s,x)|^q \xi_{r_1}(t,s,x) \, dx \, ds  \gtrsim  \int_0^t \langle s \rangle^{-(r_1+1)} \int_{\mathbb{R}^n}|v(s,x)|^q  \, dx \, ds \notag \\
& \gtrsim \varepsilon ^q  \int_0^t \langle s \rangle^{-(r_1+1)+n-1-\frac{n-1}{2}q}  \, ds. \label{lower bound mathcalU intermediate}
\end{align} Also, for $t\geqslant 0$
\begin{align*}
\mathcal{U}(t)  & \gtrsim \varepsilon ^q  \int_0^t \langle s \rangle^{-(r_1+1)+n-1-\frac{n-1}{2}q}  \, ds  \gtrsim \varepsilon ^q \langle t \rangle^{-(r_1+1)-\frac{n-1}{2}q} \int_0^t \langle s \rangle^{n-1}  \, ds \\
& \gtrsim \varepsilon ^q \langle t \rangle^{-(r_1+1)-\frac{n-1}{2}q+n}.
\end{align*} Plugging the last lower bound for $\mathcal{U}$ in \eqref{OII mathcalV case Theta2}, we have for $t\geqslant \frac{3}{2}$
\begin{align}
\mathcal{V}(t) &\gtrsim \varepsilon^{pq}  \langle t \rangle^{-1} \int_0^t (t-s) \langle s\rangle^{-r_2 -(n-1)p+n+r_1 p+\left(-(r_1+1)-\frac{n-1}{2}q+n\right)p} ds\notag \\
& = \varepsilon^{pq}  \langle t \rangle^{-1} \int_0^t (t-s) \langle s\rangle^{-r_2+n-\frac{n-1}{2}pq} ds  
 =\varepsilon^{pq}  \langle t \rangle^{-1} \int_0^t (t-s) \langle s\rangle^{1+q^{-1}-\frac{n-1}{2}(pq-1)} ds\notag \\
 & = \varepsilon^{pq}  \langle t \rangle^{-1} \int_0^t (t-s) \langle s\rangle^{-1} ds \gtrsim  \varepsilon^{pq}  \langle t \rangle^{-1} \int_1^t \frac{t-s }{s} \, ds  \gtrsim  \varepsilon^{pq}  \langle t \rangle^{-1} \int_1^t \log s \, ds \notag  \\
  & \gtrsim  \varepsilon^{pq}  \langle t \rangle^{-1} \int_{\tfrac{2t}{3}}^t \log s \, ds  \gtrsim  \varepsilon^{pq}  \log \Big(\frac{2t}{3} \Big), \label{lower bound mathcalV log case Theta2}
\end{align} where we employed in the third step the actual value of $r_2$ and in the fourth one the critical condition  $\Theta_2(n,p,q)=0$.

\subsubsection{Case $\Theta_1(n,p,q)=\Theta_2(n,p,q)=0$ }

In this last case we can improve both \eqref{lower bound mathcalU log case Theta1} and \eqref{lower bound mathcalV log case Theta2} thanks to \eqref{identity r1}, \eqref{identity r2}. Indeed, combining \eqref{lower bound mathcalU intermediate} and \eqref{identity r1}, for any $t\geqslant 0$ we obtain 
\begin{align} \label{lower bound mathcalU log case Theta1Theta2}
\mathcal{U}(t) \gtrsim \varepsilon^q \int_0^t \langle s \rangle^{-1}ds \gtrsim \varepsilon^q  \log t.
\end{align} Analogously, using \eqref{lower bound mathcalV intermediate} and \eqref{identity r2},  for any $t\geqslant \frac{3}{2}$ we have
\begin{align} \label{lower bound mathcalV log case Theta1Theta2}
\mathcal{V}(t) \gtrsim \varepsilon^p \langle t\rangle^{-1} \int_0^t (t-s) \langle s \rangle^{-1}ds \gtrsim \varepsilon^p  \log \Big(\frac{2t}{3}\Big).
\end{align}

\subsection{Iterated lower bound estimates: slicing method}
\label{Subsection iteration argument slicing method}

In this section we derive iteratively a sequence of lower bound estimates for $\mathcal{U}$ or $\mathcal{V}$. Then, in Section \ref{Subsection lifespan critical case} we will employ these iterated lower bounds to prove the blow-up and to derive the upper bound for the lifespan of the local solution $(u,v)$.

 However, before starting with this iterative procedure, we summarize the estimates that we proved in Sections \ref{Subsection iteration frame} and \ref{Subsection lower bound log}.
 
 In Section \ref{Subsection iteration frame} we proved the coupled system of integral inequalities
 \begin{align}\label{OIIs mathcalU}
 & \begin{cases} \mathcal{U}(t) \geqslant C \displaystyle{\int_0^t } \langle s\rangle^{-r_1+n-1 -(n-1)q +r_2q } \big(\mathcal{V}(s)\big)^q\, ds  & \mbox{if} \ \ \Theta_1=0, \\ \mathcal{U}(t) \geqslant C \displaystyle{\int_0^t }  \langle s\rangle^{-(r_1+1) +\frac{n-1}{2}q-(n-1)(q-1)} (\log\langle s \rangle)^{-(q-1)} \big(\mathcal{V}(s)\big)^q \, ds  & \mbox{if} \ \ \Theta_2=0,\\  \mathcal{U}(t) \geqslant C  \displaystyle{\int_0^t} \langle s\rangle^{-1}(\log \langle s\rangle )^{-(q-1)} \big(\mathcal{V}(s)\big)^q \, ds  & \mbox{if} \ \ \Theta_1=\Theta_2=0,\end{cases} \\ 
  & \begin{cases} \mathcal{V}(t)  \geqslant K \langle t \rangle^{-1} \displaystyle{\int_0^t} (t-s) \langle s\rangle^{-r_2 +\frac{n-1}{2}p-(n-1)(p-1)}(\log \langle s\rangle )^{-(p-1)} \big(\mathcal{U}(s)\big)^p \, ds   & \mbox{if} \ \ \Theta_1=0, \\ \mathcal{V}(t)  \geqslant K  \langle t \rangle^{-1} \displaystyle{\int_0^t} (t-s) \langle s\rangle^{-r_2 -(n-1)p+n+r_1 p}\big(\mathcal{U}(s)\big)^p\, ds  & \mbox{if} \ \ \Theta_2=0,\\ \mathcal{V}(t)  \geqslant K \langle t \rangle^{-1} \displaystyle{\int_0^t} (t-s) \langle s\rangle^{-1}(\log \langle s\rangle )^{-(p-1)} \big(\mathcal{U}(s)\big)^p \, ds & \mbox{if} \ \ \Theta_1=\Theta_2=0,\end{cases}
 \label{OIIs mathcalV}
 \end{align} for any $t\geqslant 0$, where $C,K$ are positive constants depending on $n,p,q,R$. Let us underline that the range for the pair $(r_1,r_1)$ is implicitly fixed by the corresponding critical case according to Section \ref{Subsection iteration frame}.
 
 On the other hand,  the lower bound estimates \eqref{lower bound mathcalU log case Theta1}, \eqref{lower bound mathcalV log case Theta2}, \eqref{lower bound mathcalU log case Theta1Theta2} and \eqref{lower bound mathcalV log case Theta1Theta2} from Section \ref{Subsection lower bound log} can be summarized as follows
 \begin{align}
 \begin{cases}  \mathcal{U}(t)\geqslant \widetilde{C} \varepsilon^{pq} \log t & \mbox{if} \ \ \Theta_1=0, \\
  \mathcal{U}(t)\geqslant \widetilde{C} \varepsilon^q \log t& \mbox{if} \ \ \Theta_1=\Theta_2=0 , \end{cases} \label{lower bound mathcalU log case}
 \end{align} for any $t\geqslant 1$ and 
 \begin{align}
 \begin{cases} \mathcal{V}(t)\geqslant  \widetilde{K} \varepsilon^{pq} \log \big(\frac{2t}{3}\big) & \mbox{if} \ \ \Theta_2=0, \\
\mathcal{V}(t)\geqslant  \widetilde{K} \varepsilon^p \log \big(\frac{2t}{3}\big) & \mbox{if} \ \ \Theta_1=\Theta_2=0,  \end{cases} \label{lower bound mathcalV log case}
 \end{align} for any $t\geqslant \frac{3}{2}$, where $\widetilde{C},\widetilde{K}$ are positive constants depending on $n,p,q,R,u_0,u_1,v_0,v_1$.
 
 Now we can start with the iteration argument. As in the previous sections, we consider separately the three critical cases.
 
 \subsubsection{Case $\Theta_1(n,p,q)=0$ }
 
 Let us introduce the sequence of positive real numbers $\{\ell_j\}_{j\in \mathbb{N}}$, where $\ell_j\doteq 2-2^{-j}$, that will be use to split the time interval in the slicing method. In this case the goal is to prove that \begin{align}
\mathcal{U}(t)\geqslant C_j \big(\log\langle t\rangle\big)^{-b_j} \bigg(\log\bigg(\frac{t}{\ell_{j}}\bigg)\bigg)^{a_j} \qquad \mbox{for} \ \ t\geqslant \ell_{j} \ \ \mbox{and for any} \ \ j\in \mathbb{N}, \label{iter ineq crit case mathcalU Theta1}
\end{align} where $\{C_j\}_{j\in\mathbb{N}}$,  $\{a_j\}_{j\in\mathbb{N}}$ and  $\{b_j\}_{j\in\mathbb{N}}$ are sequences of nonnegative real numbers that we shall determine throughout the iteration argument. Thanks to \eqref{lower bound mathcalU log case} we see that \eqref{iter ineq crit case mathcalU Theta1} is satisfied for $j=0$, provided that the initial values of the sequences are given by $a_0\doteq 1 , b_0\doteq 0 $ and  $C_0\doteq \widetilde{C} \varepsilon^{pq}$. Hence, we employ an inductive argument to prove the validity of \eqref{iter ineq crit case mathcalU Theta1} for any $j\in \mathbb{N}$. We proceed now with the inductive step. Let us plug \eqref{iter ineq crit case mathcalU Theta1} in \eqref{OIIs mathcalV}, after shrinking the domain of integration, then, for $s\geqslant \ell_{j+1}$ we obtain
\begin{align*}
\mathcal{V}(s) & \geqslant K C_j^p \langle s \rangle^{-1} \int_{\ell_{j}}^s (s-\tau) \langle \tau\rangle^{-r_2 +\frac{n-1}{2}p-(n-1)(p-1)}\big(\log \langle \tau \rangle \big)^{-(p-1)-b_jp} \left(\log\left(\tfrac{\tau}{\ell_{j}}\right)\right)^{a_j p} \, d\tau \\
& \geqslant K C_j^p   \langle s\rangle^{-1-r_2 -\frac{n-1}{2}p} \big(\log \langle s\rangle \big)^{-(p-1)-b_jp} \int_{\ell_{j}}^s (s-\tau) \langle \tau\rangle^{n-1} \left(\log\left(\tfrac{\tau}{\ell_{j}}\right)\right)^{a_j p} \, d\tau \\
& \geqslant K C_j^p   \langle s\rangle^{-1-r_2 -\frac{n-1}{2}p} \big(\log \langle s\rangle \big)^{-(p-1)-b_jp} \int_{\tfrac{\ell_{j} s}{\ell_{j+1}}}^s (s-\tau)  \tau^{n-1} \left(\log\left(\tfrac{\tau}{\ell_{j}}\right)\right)^{a_j p} \, d\tau  \\
& \geqslant K C_j^p \left(\tfrac{\ell_{j}}{\ell_{j+1}}\right)^{n-1}  \langle s\rangle^{-1-r_2 -\frac{n-1}{2}p}s^{n-1} \big(\log \langle s\rangle \big)^{-(p-1)-b_jp} \left(\log\left(\tfrac{s}{\ell_{j+1}}\right)\right)^{a_j p}\int_{\tfrac{\ell_{j} s}{\ell_{j+1}}}^s (s-\tau)   \, d\tau \\
& \geqslant 2^{-1} K C_j^p \left(\tfrac{\ell_{j}}{\ell_{j+1}}\right)^{n-1} \left(1-\tfrac{\ell_{j} }{\ell_{j+1}}\right)^2 \langle s\rangle^{-1-r_2 -\frac{n-1}{2}p}s^{n+1} \big(\log \langle s\rangle \big)^{-(p-1)-b_jp} \left(\log\left(\tfrac{s}{\ell_{j+1}}\right)\right)^{a_j p}
\\
& \geqslant 2^{-2j-3n-6} K C_j^p  \langle s\rangle^{-r_2 -\frac{n-1}{2}p+n} \big(\log \langle s\rangle \big)^{-(p-1)-b_jp} \left(\log\left(\tfrac{s}{\ell_{j+1}}\right)\right)^{a_j p},
\end{align*} where in the last step we used the inequalities $2\ell_{j}\geqslant  \ell_{j+1}$, $1-\frac{\ell_{j}}{\ell_{j+1}}\geqslant 2^{-(j+2)}$ and $4 s\geqslant \langle s \rangle$ for any $s\geqslant 1$. Using this lower bound for $\mathcal{V}(s)$ in \eqref{OIIs mathcalU} and the critical relation $\Theta_1(n,p,q)=0$, for $t\geqslant \ell_{j+1}$ we get
\begin{align*}
\mathcal{U}(t) & \geqslant  2^{-2qj-3q(n+2)} C K^q C_j^{pq} \int_{\ell_{j+1}}^t \langle s\rangle^{-r_1+n-1 +q -\frac{n-1}{2}pq }   \big(\log \langle s\rangle \big)^{-q(p-1)-b_jpq} \left(\log\left(\tfrac{s}{\ell_{j+1}}\right)\right)^{a_j pq}\, ds \\
& \geqslant  2^{-2qj-3q(n+2)} C K^q C_j^{pq} \big(\log \langle t\rangle \big)^{-q(p-1)-b_jpq} \int_{\ell_{j+1}}^t \langle s\rangle^{q+\frac{1}{p} -\frac{n-1}{2}(pq-1) }    \left(\log\left(\tfrac{s}{\ell_{j+1}}\right)\right)^{a_j pq}\, ds  \\
& =  2^{-2qj-3q(n+2)} C K^q C_j^{pq} \big(\log \langle t\rangle \big)^{-q(p-1)-b_jpq} \int_{\ell_{j+1}}^t \langle s\rangle^{-1 }    \left(\log\left(\tfrac{s}{\ell_{j+1}}\right)\right)^{a_j pq}\, ds \\
& =  2^{-2qj-3q(n+2)} C K^q C_j^{pq} (a_jpq+1)^{-1}\big(\log \langle t\rangle \big)^{-q(p-1)-b_jpq}\left(\log\left(\tfrac{t}{\ell_{j+1}}\right)\right)^{a_j pq+1},
\end{align*} which is exactly \eqref{iter ineq crit case mathcalU Theta1} for $j+1$, if we define 
\begin{align*}
C_{j+1} \doteq 2^{-2qj-3q(n+2)} C K^q C_j^{pq} (a_jpq+1)^{-1}, \quad a_{j+1}\doteq a_j pq+1, \ \ b_{j+1}\doteq q(p-1)+b_jpq. 
\end{align*}

In order to derive the upper bound estimate for the life span of the solution, it is convenient to derive an estimate from below of $C_j$, where the dependence on $j$ in the lower bound is more explicit than the one in the definition of $C_j$ itself. But first, let us derive the explicit expression for $a_j$ and $b_j$. Using iteratively the recursive relations between two successive elements that we just proved, we find
\begin{equation} \label{representation aj,bj crit case}
\begin{split}
a_j &=a_{j-1}pq+1= \cdots = a_0(pq)^j +\sum_{k=0}^{j-1} (pq)^k= (pq)^j +\tfrac{(pq)^j-1}{pq-1}=\tfrac{(pq)^{j+1}-1}{pq-1}, \\
b_j &=b_{j-1}pq+q(p-1) = \cdots = b_0(pq)^j +q(p-1)\sum_{k=0}^{j-1} (pq)^k=\tfrac{q(p-1)}{pq-1}\big((pq)^j-1\big).
\end{split}
\end{equation} Therefore,
\begin{align*}
a_{j-1}pq+1\leqslant \tfrac{pq}{pq-1}(pq)^j.
\end{align*} In particular, the previous inequality implies
\begin{align}\label{lower bound Cj crit}
C_j \geqslant M N^{-j} C_{j-1}^{pq},
\end{align} where $M\doteq 2^{-q(3n+4)} C K^q\tfrac{(pq-1)}{pq}$ and $N \doteq 2^{2q}pq$. 
Applying the logarithmic function to both sides of \eqref{lower bound Cj crit} and using iteratively the resulting inequality, we get
\begin{align*}
\log C_j & \geqslant (pq) \log C_{j-1} -j\log N+\log M \\
& \geqslant (pq)^2 \log C_{j-2} -\big(j+(j-1)(pq)\big)\log N+\big(1+pq\big)\log M \\
& \geqslant \cdots \geqslant  (pq)^j \log C_{0} -\sum_{k=0}^{j-1}(j-k)(pq)^k\log N+\sum_{k=0}^{j-1} (pq)^k\log M \\
& = (pq)^j \log C_{0} -(pq)^j \sum_{k=1}^{j}k(pq)^{-k} \log N+\tfrac{(pq)^j-1}{pq-1}\log M  \\ &= (pq)^j \bigg(\log C_{0} -S_j\log N+\tfrac{\log M }{pq-1}\bigg)-\tfrac{\log M }{pq-1},
\end{align*} where $S_j\doteq \sum_{k=1}^j k(pq)^{-k}$. As $\{S_j\}_{j\geqslant 1}$ is a sequence of the partial sums of a convergent series, if we denote by $S$ the limit of this sequence, because of $S_j\uparrow S$ as $j\to \infty$, then, we may estimate
\begin{align}\label{lower bound Cj crit exp}
C_j\geqslant M^{-(pq-1)}\exp\Big((pq)^j\log\Big(C_0 N^{-S} M^{pq-1}\Big)\Big).
\end{align}

 \subsubsection{Case $\Theta_2(n,p,q)=0$ }
 
 In this second critical case we shall prove that 
 \begin{align}
\mathcal{V}(t)\geqslant K_j \big(\log\langle t\rangle\big)^{-\beta_j} \bigg(\log\bigg(\frac{t}{\ell_{2j+1}}\bigg)\bigg)^{\alpha_j} \qquad \mbox{for} \ \ t\geqslant \ell_{2j+1} \ \ \mbox{and for any} \ \ j\in \mathbb{N}, \label{iter ineq crit case mathcalV Theta2}
\end{align} where $\{K_j\}_{j\in\mathbb{N}}$,  $\{\alpha_j\}_{j\in\mathbb{N}}$ and  $\{\beta_j\}_{j\in\mathbb{N}}$ are sequences of nonnegative real numbers that we will be fixed during the iterative procedure.  Due to \eqref{lower bound mathcalV log case} we see that \eqref{iter ineq crit case mathcalV Theta2} is satisfied for $j=0$, supposed that the initial values of the sequences are given by $\alpha_0\doteq 1 , \beta_0\doteq 0 $ and  $K_0\doteq \widetilde{K} \varepsilon^{pq}$. Also in this case it remains to prove the inductive step in order to show the validity of \eqref{iter ineq crit case mathcalV Theta2} for any $j\in\mathbb{N}$. For this purpose we plug in \eqref{iter ineq crit case mathcalV Theta2} in \eqref{OIIs mathcalU}, so that, after a restriction of the domain of integration, for $s\geqslant \ell_{2j+2}$ we have
\begin{align*}
\mathcal{U}(s) & \geqslant C K_j^q \int_{\ell_{2j+1}}^s \langle \tau\rangle^{-(r_1+1) +\frac{n-1}{2}q-(n-1)(q-1)} \big(\log\langle \tau \rangle\big)^{-(q-1)-\beta_jq}   \left(\log\left(\tfrac{\tau}{\ell_{2j+1}}\right)\right)^{\alpha_j q} \, d\tau \\
 & \geqslant C K_j^q \langle s\rangle^{-(r_1+1) -\frac{n-1}{2}q} \big(\log\langle s \rangle\big)^{-(q-1)-\beta_jq}  \int_{\ell_{2j+1}}^s   \langle \tau\rangle^{n-1} \left(\log\left(\tfrac{\tau}{\ell_{2j+1}}\right)\right)^{\alpha_j q} \, d\tau \\
 & \geqslant C K_j^q \langle s\rangle^{-(r_1+1) -\frac{n-1}{2}q} \big(\log\langle s \rangle\big)^{-(q-1)-\beta_jq}  \int_{\tfrac{\ell_{2j+1}s}{\ell_{2j+2}}}^s   \tau^{n-1} \left(\log\left(\tfrac{\tau}{\ell_{2j+1}}\right)\right)^{\alpha_j q} \, d\tau \\ 
 & \geqslant C K_j^q  \left(\tfrac{\ell_{2j+1}}{\ell_{2j+2}}\right)^{n-1}\left(1-\tfrac{\ell_{2j+1}}{\ell_{2j+2}}\right)\langle s\rangle^{-(r_1+1) -\frac{n-1}{2}q} s^{n} \big(\log\langle s \rangle\big)^{-(q-1)-\beta_jq}  \left(\log\left(\tfrac{s}{\ell_{2j+2}}\right)\right)^{\alpha_j q} \\
 & \geqslant 2^{-2j-3n-2} C K_j^q  \langle s\rangle^{-(r_1+1) -\frac{n-1}{2}q+n} \big(\log\langle s \rangle\big)^{-(q-1)-\beta_jq}  \left(\log\left(\tfrac{s}{\ell_{2j+2}}\right)\right)^{\alpha_j q}.
\end{align*}
Combining this lower bound for $\mathcal{U}(s)$ and \eqref{OIIs mathcalV} and using the critical relation $\Theta_2(n,p,q)=0$, for $t\geqslant \ell_{2j+3}$ we arrive at
\begin{align*}
\mathcal{V}(t) & \geqslant 2^{-2pj-(3n+2)p}  K C^p K_j^{pq}   \langle t \rangle^{-1} \int_{ \ell_{2j+2}}^t (t-s)  \langle s\rangle^{-r_2 +n -\frac{n-1}{2}pq} \big(\log\langle s \rangle\big)^{-p(q-1)-\beta_jq}  \left(\log\left(\tfrac{s}{\ell_{2j+2}}\right)\right)^{\alpha_j pq}\, ds \\
& \geqslant 2^{-2pj-(3n+2)p}  K C^p K_j^{pq} \big(\log\langle t \rangle\big)^{-p(q-1)-\beta_jq}     \langle t \rangle^{-1} \int_{ \ell_{2j+2}}^t (t-s)  \langle s\rangle^{1+\frac{1}{q}-\frac{n-1}{2}(pq-1)} \left(\log\left(\tfrac{s}{\ell_{2j+2}}\right)\right)^{\alpha_j pq}\, ds \\
& = 2^{-2pj-(3n+2)p}  K C^p K_j^{pq} \big(\log\langle t \rangle\big)^{-p(q-1)-\beta_jq}     \langle t \rangle^{-1} \int_{ \ell_{2j+2}}^t (t-s)  \langle s\rangle^{-1} \left(\log\left(\tfrac{s}{\ell_{2j+2}}\right)\right)^{\alpha_j pq}\, ds \\
& \geqslant 2^{-2pj-(3n+2)p-2}  K C^p K_j^{pq} \big(\log\langle t \rangle\big)^{-p(q-1)-\beta_jq}     \langle t \rangle^{-1} \int_{ \ell_{2j+2}}^t \tfrac{t-s}{s} \left(\log\left(\tfrac{s}{\ell_{2j+2}}\right)\right)^{\alpha_j pq}\, ds \\
& = 2^{-2pj-(3n+2)p-2}  K C^p K_j^{pq} (\alpha_j pq+1)^{-1} \big(\log\langle t \rangle\big)^{-p(q-1)-\beta_jq}     \langle t \rangle^{-1} \int_{ \ell_{2j+2}}^t  \left(\log\left(\tfrac{s}{\ell_{2j+2}}\right)\right)^{\alpha_j pq+1}\, ds \\
& \geqslant 2^{-2pj-(3n+2)p-2}  K C^p K_j^{pq} (\alpha_j pq+1)^{-1} \big(\log\langle t \rangle\big)^{-p(q-1)-\beta_jq}     \langle t \rangle^{-1} \int_{\tfrac{ \ell_{2j+2}t}{\ell_{2j+3}}}^t  \left(\log\left(\tfrac{s}{\ell_{2j+2}}\right)\right)^{\alpha_j pq+1}\, ds \\
& \geqslant 2^{-2pj-(3n+2)p-4}  K C^p K_j^{pq} (\alpha_j pq+1)^{-1} \left(1-\tfrac{ \ell_{2j+2}}{\ell_{2j+3}}\right)\big(\log\langle t \rangle\big)^{-p(q-1)-\beta_jq}     \left(\log\left(\tfrac{t}{\ell_{2j+3}}\right)\right)^{\alpha_j pq+1} \\
& \geqslant 2^{-2(p+1)j-(3n+2)p-8}  K C^p K_j^{pq} (\alpha_j pq+1)^{-1} \big(\log\langle t \rangle\big)^{-p(q-1)-\beta_jq}     \left(\log\left(\tfrac{t}{\ell_{2j+3}}\right)\right)^{\alpha_j pq+1},
\end{align*} that is \eqref{iter ineq crit case mathcalV Theta2} for $j+1$, provided that 
\begin{align*} K_{j+1}\doteq 2^{-2(p+1)j-(3n+2)p-8}  K C^p K_j^{pq} (\alpha_j pq+1)^{-1}, \quad \alpha_{j+1}\doteq \alpha_j pq+1, \ \ \beta_{j+1}\doteq \beta_jq+ p(q-1).
\end{align*}  
Analogously to what we did  in the first critical case, we derive now a lower bound for $K_j$. Let us find first the expression of $\alpha_j$ and $\beta_j$. Applying iteratively the definitions of $\alpha_j$ and $\beta_j$, we end up with the representation formulas
\begin{equation} \label{representation alphaj,betaj crit case}
\begin{split}
\alpha_j & =\alpha_{j-1} pq+1 = \cdots = \alpha_0 (pq)^j +\tfrac{(pq)^j-1}{pq-1}= \tfrac{(pq)^{j+1}-1}{pq-1}, \\
\beta_j & =\beta_{j-1} pq+p(q-1) = \cdots = \beta_0 (pq)^j +\tfrac{p(q-1)}{pq-1}\big((pq)^j-1\big)= \tfrac{p(q-1)}{pq-1}\big((pq)^j-1\big).
\end{split}
\end{equation} In particular, it holds the inequality $\alpha_{j-1} pq+1\leqslant \tfrac{pq}{pq-1}(pq)^{j} $, that implies in turn
\begin{align}\label{lower bound Kj}
K_j \geqslant M_1 N_1^{-j} K_{j-1}^{pq},
\end{align} where $M_1\doteq 2^{-3np-6}  K C^p \frac{(pq-1)}{pq}$ and $N_1\doteq 2^{2(p+1)}pq$.
Analogously to the derivation of \eqref{lower bound Cj crit exp} via \eqref{lower bound Cj crit}, from \eqref{lower bound Kj} we obtain
\begin{align}\label{lower bound Kj exp}
K_j\geqslant M_1^{-(pq-1)}\exp\Big((pq)^j\log\Big(K_0 N_1^{-S} M_1^{pq-1}\Big)\Big).
\end{align}

 \subsubsection{Case $\Theta_1(n,p,q)= \Theta_2(n,p,q)=0$ }
 
 In this third critical case the goal is to prove that
 \begin{align}
\mathcal{U}(t)\geqslant D_j \big(\log\langle t\rangle\big)^{-h_j} \bigg(\log\bigg(\frac{t}{\ell_{j}}\bigg)\bigg)^{g_j} \qquad \mbox{for} \ \ t\geqslant \ell_{j} \ \ \mbox{and for any} \ \ j\in \mathbb{N}, \label{iter ineq crit case mathcalU Theta1Theta2}
\end{align} where $\{D_j\}_{j\in\mathbb{N}}$,  $\{g_j\}_{j\in\mathbb{N}}$ and  $\{h_j\}_{j\in\mathbb{N}}$ are sequences of nonnegative real numbers that we shall determine throughout the iteration argument. Due to the different iteration scheme for $\Theta_1(n,p,q)= \Theta_2(n,p,q)=0$ in \eqref{OIIs mathcalU} and \eqref{OIIs mathcalV}, the proof of the inductive step will have somehow a more symmetric behavior than the ones in the previous cases. We point out that \eqref{lower bound mathcalU log case} implies the validity of \eqref{iter ineq crit case mathcalU Theta1Theta2} in the base case $j=0$ if we consider
\begin{align*}
D_0\doteq \widetilde{C} \varepsilon^q, \quad g_0\doteq 1, \ \ h_0\doteq 0.
\end{align*}

Let us proceed with the inductive step. If we plug \eqref{iter ineq crit case mathcalU Theta1Theta2} in \eqref{OIIs mathcalV}, then for $s\geqslant \ell_{j+1}$ we get
 \begin{align*}
 \mathcal{V}(s) & \geqslant K  D_j^p  \langle s \rangle^{-1} \int_{\ell_{j}}^s (s-\tau) \langle \tau \rangle^{-1}\big(\log \langle \tau\rangle \big)^{-(p-1)-h_jp}  \left( \log\left(\tfrac{\tau}{\ell_{j}}\right)\right)^{g_j p}\, d\tau \\
 & \geqslant 2^{-2} K  D_j^p  \big(\log \langle s\rangle \big)^{-(p-1)-h_jp} \langle s \rangle^{-1}  \int_{\ell_{j}}^s \tfrac{s-\tau}{ \tau} \left( \log\left(\tfrac{\tau}{\ell_{j}}\right)\right)^{g_j p}\, d\tau \\
 & = 2^{-2} K  D_j^p (g_jp+1)^{-1}  \big(\log \langle s\rangle \big)^{-(p-1)-h_jp}  \langle s \rangle^{-1} \int_{\ell_{j}}^s  \left( \log\left(\tfrac{\tau}{\ell_{j}}\right)\right)^{g_j p+1}\, d\tau  \\
 & \geqslant 2^{-2} K  D_j^p (g_jp+1)^{-1}  \big(\log \langle s\rangle \big)^{-(p-1)-h_jp}  \langle s \rangle^{-1} \int_{\tfrac{\ell_{j}s}{\ell_{j+1}}}^s  \left( \log\left(\tfrac{\tau}{\ell_{j}}\right)\right)^{g_j p+1}\, d\tau \\
  & \geqslant 2^{-4} K  D_j^p (g_jp+1)^{-1}  \left(1-\tfrac{\ell_{j}}{\ell_{j+1}}\right) \big(\log \langle s\rangle \big)^{-(p-1)-h_jp}  \left( \log\left(\tfrac{s}{\ell_{j+1}}\right)\right)^{g_j p+1} \\
  & \geqslant 2^{-(j+6)} K  D_j^p (g_jp+1)^{-1}  \big(\log \langle s\rangle \big)^{-(p-1)-h_jp}  \left( \log\left(\tfrac{s}{\ell_{j+1}}\right)\right)^{g_j p+1}.
 \end{align*}
A combination of the previous lower bound for $\mathcal{V}(s)$ and \eqref{OIIs mathcalU} yields for $t\geqslant \ell_{j+1}$
\begin{align*}
\mathcal{U}(t) & \geqslant 2^{-(j+6)q} C K^q  D_j^{pq} (g_jp+1)^{-q} \int_{\ell_{j+1}}^t \langle s\rangle^{-1}   \big(\log \langle s\rangle \big)^{-(pq-1)-h_jpq}  \left( \log\left(\tfrac{s}{\ell_{j+1}}\right)\right)^{g_j pq+q} \, ds \\
& \geqslant 2^{-(j+6)q-2} C K^q  D_j^{pq} (g_jp+1)^{-q}  \big(\log \langle t\rangle \big)^{-(pq-1)-h_jpq}  \int_{\ell_{j+1}}^t  s^{-1}   \left( \log\left(\tfrac{s}{\ell_{j+1}}\right)\right)^{g_j pq+q} \, ds \\
& = 2^{-(j+6)q-2} C K^q  D_j^{pq} (g_jp+1)^{-q} (g_j pq+q+1)^{-1} \big(\log \langle t\rangle \big)^{-(pq-1)-h_jpq}     \left( \log\left(\tfrac{t}{\ell_{j+1}}\right)\right)^{g_j pq+q+1}. 
\end{align*} The last inequality is \eqref{iter ineq crit case mathcalU Theta1Theta2} in the case $j+1$ with
\begin{align*}
D_{j+1}\doteq 2^{-(j+6)q-2} C K^q  D_j^{pq} (g_jp+1)^{-q} (g_j pq+q+1)^{-1} , \quad g_{j+1}\doteq g_j pq+q+1, \ \  h_{j+1}\doteq h_jpq+ pq-1.
\end{align*} Finally, we find a lower bound for the coefficient $D_j$. First, we have
\begin{equation} \label{representation gj,hj crit case}
\begin{split}
g_j & = g_{j-1} pq+q+1= \cdots = g_0 (pq)^j +(q+1) \sum_{k=0}^{j-1} (pq)^k= \left(1+\tfrac{q+1}{pq-1}\right)(pq)^j- \tfrac{q+1}{pq-1}, \\
h_j & = h_{j-1} pq+pq-1= \cdots = h_0 (pq)^j +(pq-1) \sum_{k=0}^{j-1} (pq)^k= (pq)^j-1.
\end{split}
\end{equation}
Hence, $$g_{j-1} p+1\leqslant g_{j-1} pq+q+1\leqslant \tfrac{q(p+1)}{pq-1}(pq)^j$$ implies 
\begin{align}\label{lower bound Dj}
D_j \geqslant M_2N_2^{-j} D_{j-1}^{pq},
\end{align} where $M_2\doteq 2^{-5q-2} C K^q \big(\tfrac{pq-1}{q(p+1)}\big)^{q+1}$ and $N_2\doteq 2^q (pq)^{q+1}$. In an analogous way as in the derivation of \eqref{lower bound Cj  crit exp} through \eqref{lower bound Cj crit}, by \eqref{lower bound Dj} we have
\begin{align}\label{lower bound Dj exp}
D_j\geqslant M_2^{-(pq-1)}\exp\Big((pq)^j\log\Big(D_0 N_2^{-S} M_2^{pq-1}\Big)\Big).
\end{align}

\subsection{Upper bound for the lifespan of local solutions} \label{Subsection lifespan critical case}

In this section we finally prove that a local solution $(u,v)$ of \eqref{weakly coupled system crit} blows up in finite time under the assumption of Theorem \ref{Thm crit case}. As in the previous sections we will consider separately the three critical cases.

\subsubsection{Case $\Theta_1(n,p,q)=0$ }

If we combine \eqref{iter ineq crit case mathcalU Theta1}, \eqref{representation aj,bj crit case} and \eqref{lower bound Cj crit exp}, then, for $t\geqslant 2\geqslant \ell_j$ we have
\begin{align*}
\mathcal{U}(t) & \geqslant M^{-(pq-1)}\exp\Big((pq)^j\log\big(C_0 N^{-S} M^{pq-1}\big)\Big)\big(\log\langle t\rangle\big)^{-b_j} \Big(\log\big(\tfrac{t}{\ell_{j}}\big)\Big)^{a_j} \\
&  \geqslant M^{-(pq-1)}\exp\Big((pq)^j\log\big(C_0 N^{-S} M^{pq-1}\big)\Big)\big(\log\langle t\rangle\big)^{-b_j} \Big(\log\big(\tfrac{t}{2}\big)\Big)^{a_j} \\
&  \geqslant M^{-(pq-1)} \bigg(\tfrac{\left(\log\langle t\rangle\right)^{q(p-1)}}{\log\left(\tfrac{t}{2}\right)}\bigg)^{\frac{1}{pq-1}}\exp\Big((pq)^j\log\Big(\big(C_0 N^{-S} M^{pq-1}\big)-\tfrac{q(p-1)}{pq-1}\log\langle t \rangle +\tfrac{pq}{pq-1}\log\big(\tfrac{t}{2}\big)\Big)\Big). 
\end{align*} Because for $t\geqslant 4$ the inequalities $\log\langle t \rangle \leqslant \log (2t)\leqslant 2\log t$ and $\log\big(\tfrac{t}{2}\big)\geqslant \tfrac{1}{2}\log t$ hold, then for $t\geqslant 4$ the last estimate from below for $\mathcal{U}(t)$ implies
\begin{align}
\mathcal{U}(t) & \geqslant M^{-(pq-1)} \bigg(\tfrac{\left(\log\langle t\rangle\right)^{q(p-1)}}{\log\left(\tfrac{t}{2}\right)}\bigg)^{\frac{1}{pq-1}}\exp\Big((pq)^j\log\Big(2^{-\frac{q(2p-1)}{pq-1}} C_0 N^{-S} M^{pq-1}(\log t)^{\frac{q}{pq-1}}\Big)\Big)\notag \\
& \geqslant M^{-(pq-1)} \bigg(\tfrac{\left(\log\langle t\rangle\right)^{q(p-1)}}{\log\left(\tfrac{t}{2}\right)}\bigg)^{\frac{1}{pq-1}}\exp\Big((pq)^j\log\Big(E\varepsilon^{pq}(\log t)^{\frac{q}{pq-1}}\Big)\Big), \label{lower bound mathcalU exp log final}
\end{align} where $E\doteq 2^{-\frac{q(2p-1)}{pq-1}} \widetilde{C} N^{-S} M^{pq-1}$. \newline Let us point out that $H(t,\varepsilon)\doteq E\varepsilon^{pq}(\log t)^{\frac{q}{pq-1}}>1$ if and only if $t>\exp \big(E^{-\frac{pq-1}{q}}\varepsilon^{-p(pq-1)}\big)$. Consequently, we can fix a sufficiently small $\varepsilon_0$ such that $$\exp \big(E^{-\frac{pq-1}{q}}\varepsilon_0^{-p(pq-1)}\big)\geqslant 4.$$ Then, for any $\varepsilon\in (0,\varepsilon_0]$ and $t>\exp \big(E^{-\frac{pq-1}{q}}\varepsilon^{-p(pq-1)}\big) \geqslant 4$ it holds $H(t,\varepsilon)>1$; so, letting $j\to \infty$ in \eqref{lower bound mathcalU exp log final}, we find that the lower bound for $\mathcal{U}(t)$ blows up. Thus, we proved that $\mathcal{U}(t)$ can be finite only for 
\begin{align}\label{lifespan crit case Theta1}
t\leqslant \exp \big(E^{-\frac{pq-1}{q}}\varepsilon^{-p(pq-1)}\big),
\end{align}
 which is the upper bound estimate for the lifespan in \eqref{lifespan upper bound estimate crit} when $\Theta_1(n,p,q)=0$.

\subsubsection{Case $\Theta_2(n,p,q)=0$ }

Combining \eqref{iter ineq crit case mathcalV Theta2}, \eqref{representation alphaj,betaj crit case} and \eqref{lower bound Kj exp}, then, for $t\geqslant 2\geqslant \ell_{2j+1}$ we have
\begin{align*}
\mathcal{V}(t) & \geqslant M_1^{-(pq-1)}\exp\Big((pq)^j\log\big(K_0 N_1^{-S} M_1^{pq-1}\big)\Big)\big(\log\langle t\rangle\big)^{-\beta_j} \Big(\log\big(\tfrac{t}{\ell_{2j+1}}\big)\Big)^{\alpha_j} \\
&  \geqslant M_1^{-(pq-1)}\exp\Big((pq)^j\log\big(K_0 N_1^{-S} M_1^{pq-1}\big)\Big)\big(\log\langle t\rangle\big)^{-\beta_j} \Big(\log\big(\tfrac{t}{2}\big)\Big)^{\alpha_j} \\
&  \geqslant M_1^{-(pq-1)} \bigg(\tfrac{\left(\log\langle t\rangle\right)^{p(q-1)}}{\log\left(\tfrac{t}{2}\right)}\bigg)^{\frac{1}{pq-1}}\exp\Big((pq)^j\log\Big(\big(K_0 N_1^{-S} M_1^{pq-1}\big)-\tfrac{p(q-1)}{pq-1}\log\langle t \rangle +\tfrac{pq}{pq-1}\log\big(\tfrac{t}{2}\big)\Big)\Big). 
\end{align*} Analogously as in the last section, for $t\geqslant 4$ this estimate from below for $\mathcal{V}(t)$ provides
\begin{align}
\mathcal{V}(t) & \geqslant M_1^{-(pq-1)} \bigg(\tfrac{\left(\log\langle t\rangle\right)^{p(q-1)}}{\log\left(\tfrac{t}{2}\right)}\bigg)^{\frac{1}{pq-1}}\exp\Big((pq)^j\log\Big(2^{-\frac{p(2q-1)}{pq-1}} K_0 N_1^{-S} M_1^{pq-1}(\log t)^{\frac{p}{pq-1}}\Big)\Big)\notag \\
& \geqslant M_1^{-(pq-1)} \bigg(\tfrac{\left(\log\langle t\rangle\right)^{p(q-1)}}{\log\left(\tfrac{t}{2}\right)}\bigg)^{\frac{1}{pq-1}}\exp\Big((pq)^j\log\Big(E_1\varepsilon^{pq}(\log t)^{\frac{p}{pq-1}}\Big)\Big), \label{lower bound mathcalV exp log final}
\end{align} where $E_1\doteq 2^{-\frac{p(2q-1)}{pq-1}} \widetilde{K} N_1^{-S} M_1^{pq-1}$.  If we denote $H_1(t,\varepsilon)\doteq E_1\varepsilon^{pq}(\log t)^{\frac{p}{pq-1}}$, then, $H_1(t,\varepsilon)>1$ if and only if $t>\exp \big(E_1^{-\frac{pq-1}{p}}\varepsilon^{-q(pq-1)}\big)$. Therefore, we can choose a sufficiently small $\varepsilon_0$ such that $\exp \big(E_1^{-\frac{pq-1}{p}}\varepsilon_0^{-q(pq-1)}\big)\geqslant 4$. Also, for any $\varepsilon\in (0,\varepsilon_0]$ and $t>\exp \big(E_1^{-\frac{pq-1}{p}}\varepsilon^{-q(pq-1)}\big)$ we have $H_1(t,\varepsilon)>1$; thus, taking the limit as $j\to \infty$ in \eqref{lower bound mathcalV exp log final}, we find that the lower bound for $\mathcal{V}(t)$ blows up. Hence, we showed that $\mathcal{V}(t)$ may be finite just for 
\begin{align}\label{lifespan crit case Theta2}
t\leqslant \exp \big(E_1^{-\frac{pq-1}{p}}\varepsilon^{-q(pq-1)}\big),
\end{align} which is exactly the upper bound estimate for the lifespan in \eqref{lifespan upper bound estimate crit} for $\Theta_2(n,p,q)=0$.

 \subsubsection{Case $\Theta_1(n,p,q)= \Theta_2(n,p,q)=0$ }
For $t\geqslant 2\geqslant \ell_{2j+1}$ the combination of \eqref{iter ineq crit case mathcalU Theta1Theta2}, \eqref{representation gj,hj crit case} and \eqref{lower bound Dj exp} leads to
\begin{align*}
\mathcal{U}(t) & \geqslant M_2^{-(pq-1)}\exp\Big((pq)^j\log\big(D_0 N_2^{-S} M_2^{pq-1}\big)\Big)\big(\log\langle t\rangle\big)^{-h_j} \Big(\log\big(\tfrac{t}{\ell_{j}}\big)\Big)^{g_j} \\
&  \geqslant M_2^{-(pq-1)}\exp\Big((pq)^j\log\big(D_0 N_2^{-S} M_2^{pq-1}\big)\Big)\big(\log\langle t\rangle\big)^{-h_j} \Big(\log\big(\tfrac{t}{2}\big)\Big)^{g_j} \\
&  \geqslant M_2^{-(pq-1)} \log\langle t\rangle\left(\log\left(\tfrac{t}{2}\right)\right)^{-\frac{q+1}{pq-1}}\exp\Big((pq)^j\log\Big(\big(D_0 N_2^{-S} M_2^{pq-1}\big)-\log\langle t \rangle +\left(1+\tfrac{q+1}{pq-1}\right)\log\big(\tfrac{t}{2}\big)\Big)\Big). 
\end{align*} Similarly to the last sections, for $t\geqslant 4$ the above estimate from below for $\mathcal{U}(t)$ yields
\begin{align}
\mathcal{V}(t) & \geqslant M_2^{-(pq-1)} \log\langle t\rangle\left(\log\left(\tfrac{t}{2}\right)\right)^{-\frac{q+1}{pq-1}}\exp\Big((pq)^j\log\Big(2^{-2-\frac{q+1}{pq-1}} D_0 N_2^{-S} M_2^{pq-1}(\log t)^{\frac{q+1}{pq-1}}\Big)\Big)\notag \\
& \geqslant M_2^{-(pq-1)} \log\langle t\rangle\left(\log\left(\tfrac{t}{2}\right)\right)^{-\frac{q+1}{pq-1}} \exp\Big((pq)^j\log\Big(E_2\varepsilon^{q}(\log t)^{\frac{q+1}{pq-1}}\Big)\Big), \label{lower bound mathcalV exp log final Theta1Theta2}
\end{align} where $E_2\doteq 2^{-2-\frac{q+1}{pq-1}} \widetilde{C} N_2^{-S} M_2^{pq-1}$.  Let us denote $H_2(t,\varepsilon)\doteq E_2\varepsilon^{q}(\log t)^{\frac{q+1}{pq-1}}$. Then, $H_2(t,\varepsilon)>1$ if and only if $t>\exp \big(E_2^{-\frac{pq-1}{q+1}}\varepsilon^{-\frac{q}{q+1}(pq-1)}\big)$. Therefore, as before we can choose a sufficiently small $\varepsilon_0$ such that $\exp \big(E_2^{-\frac{pq-1}{q+1}}\varepsilon_0^{-\frac{q}{q+1}(pq-1)}\big)\geqslant 4$. Also, for any $\varepsilon\in (0,\varepsilon_0]$ and $t>\exp \big(E_2^{-\frac{pq-1}{q+1}}\varepsilon^{-\frac{q}{q+1}(pq-1)}\big)$ we have $H_2(t,\varepsilon)>1$; thus, taking the limit as $j\to \infty$ in \eqref{lower bound mathcalV exp log final Theta1Theta2}, we see that the lower bound for $\mathcal{U}(t)$ diverges. So, we proved that if $\mathcal{U}(t)$ is finite, then,  
\begin{align} \label{lifespan crit case Theta1Theta2}
t\leqslant \exp \big(E_2^{-\frac{pq-1}{q+1}}\varepsilon^{-\frac{q}{q+1}(pq-1)}\big),
\end{align} that is, we proved \eqref{lifespan upper bound estimate crit} in the critical case $\Theta_1(n,p,q)=\Theta_2(n,p,q)=0$.

\subsection{Final remarks on the critical case} \label{Subsection comparison with ISW}

\subsubsection{Comparison with other results}

As we have already mentioned in the introduction, Ikeda-Sobajima-Wakasa very recently proved a blow-up result for the semilinear weakly coupled system \eqref{weakly coupled system crit} both in the subcritical case and in the critical case, by using a revised test function method. While in the subcritical case we obtained exactly the same result (but including damping terms in the scattering case), in the critical case we got quite different estimates for the lifespan in all three subcases. Let us compare our results with theirs.

In the first critical case $\Theta_1(n,p,q)=0$ we proved the estimate \eqref{lifespan crit case Theta1}, while  in \cite{ISW18} the upper bound estimate
\begin{align} \label{lifespan crit case Theta1 ISW}
T(\varepsilon)\leqslant \exp\big(C\varepsilon^{-q(pq-1)}\big)
\end{align} is proved. Let us point out that in the critical case $\Theta_1(n,p,q)=0>\Theta_1(n,p,q) $  it is not possible to determine, in general, which exponent among $p$ and $q$ is the biggest one. So far, the best estimate for the lifespan that we can get is the one obtained combining \eqref{lifespan crit case Theta1} and \eqref{lifespan crit case Theta1 ISW}, that is,
\begin{align*}
T(\varepsilon)\leqslant \exp\left(C\varepsilon^{-\min\{p(pq,1),q(pq-1)\}}\right) \quad \mbox{if} \ \  \Theta_1(n,p,q)=0.
\end{align*}

On the contrary, in the case $\Theta_2(n,p,q)=0$ we obtained \eqref{lifespan crit case Theta2}, which is an improvement of the estimate $T(\varepsilon)\leqslant \exp\big(C\varepsilon^{-p(pq-1)}\big)$ proved in \cite{ISW18} in the same critical case. Indeed, in this case we have 
\begin{align*}
\tfrac{q+1+p^{-1}}{pq-1}-\tfrac{n-1}{2}<0=\tfrac{2+q^{-1}}{pq-1}-\tfrac{n-1}{2}
\end{align*} which provides $q-q^{-1}<1-p^{-1}<p-p^{-1}$, that implies in turn $q<p$.

We consider now the case $\Theta_1(n,p,q)= \Theta_2(n,p,q)=0$ . We point out explicitly that in this critical case we could employ an iteration argument for the functional $\mathcal{V}$ as well in the last section. Nevertheless, we would find as upper bound for the lifespan $$T(\varepsilon)\leqslant \exp\big(C\varepsilon^{-\frac{p}{p+1}(pq-1)}\big)$$ which is weaker than the one that we derived by working with $\mathcal{U}$, namely,  \eqref{lifespan crit case Theta1Theta2}. This is due to the comparison of the two critical conditions $\Theta_1(n,p,q)=0$ and $ \Theta_2(n,p,q)=0$ that lead to $q-q^{-1}=1-p^{-1}<p-p^{-1}$, which implies as above $q<p$. Morever, we emphasize that we have improved the estimate for this case in comparison to the on in \cite{ISW18} for the corresponding case, namely, $T(\varepsilon)\leqslant \exp\big(C\varepsilon^{-(pq-1)}\big)$.

\subsubsection{The intersection point of the critical curves}


Finally, we remark that in the critical case $\Theta_1(n,p,q)=\Theta_2(n,p,q)=0$, we can determine the expression of $p$ and $q$, that is, we determine the coordinates of the intersection point of the critical curves  in the $p$ - $q$ plane. By straightforward calculations, we get that $\Theta_1(n,p,q)=\Theta_2(n,p,q)$ implies
\begin{align}\label{p expression}
p= (1+q^{-1}-q)^{-1}.
\end{align} We underline that we should require $1<q<\frac{1+\sqrt{5}}{2}$, in order to get an admissible $p$. If we plug in \eqref{p expression} in $\Theta_2(n,p,q)=0$, we find that $q$ satisfies the cubic equation
\begin{align}
0 & =(n+1)q^3-\tfrac{n+1}{2}q^2-\tfrac{n+5}{2}q-1 \notag \\
& =(2q+1)\Big(\tfrac{n+1}{2}q^2-\tfrac{n+1}{2}q-1\Big). \label{cubic q}
\end{align} Therefore, the only admissible solution of \eqref{cubic q} is 
\begin{align*}
q_{\mix}(n)\doteq \frac{1}{2}\left(1+\sqrt{\frac{n+9}{n+1}}\right).
\end{align*} It is easy to check that $q_{\mix}(n)<\frac{1+\sqrt{5}}{2}$ for any $n\geqslant 2$. Plugging this expression for $q_{\mix}(n)$ in \eqref{p expression}, we get 
\begin{align*}
p_{\mix}(n) & \doteq \frac{q_{\mix}(n)}{1+q_{\mix}(n)-(q_{\mix}(n))^2} \\ &=\frac{n+1+\sqrt{(n+9)(n+1)}}{2(n-1)}.
\end{align*}

It is interesting to compare these exponents, $p_{\mix}(n)$ and $q_{\mix}(n)$, with the critical exponent for the semilinear wave equation with power nonlinearity, i.e., the Strauss exponent
$$p_{\Str}(n)=\frac{n+1+\sqrt{n^2+10n-7}}{2(n-1)}$$ and with the exponent for the semilinear wave equation of derivative type, i.e., the Glassey exponent $$p_{\Gla}(n)=\frac{n+1}{n-1}.$$

Elementary computations show that $$q_{\mix}(n) < p_{\Gla}(n)<p_{\Str}(n)<p_{\mix}(n)$$ for any $n\geqslant 2$. 
Therefore, we may conclude that for the cusp point of the critical curve for \eqref{weakly coupled system crit} the power of the nonlinear term $|\partial_t u|^p$ is bigger than the critical power for the semilinear wave equation of derivative type, while the power of the nonlinear term $|v|^q$ is smaller than the critical power for the semilinear wave equation with power nonlinearity. In this sense, we have a balance between $p$ and $q$ for the cusp point of the critical curve for the weakly coupled system of semilinear wave equations with mixed nonlinear terms, in comparison to the cases with power nonlinearities and of derivative type.

\section*{Acknowledgments}

The first author is member of the Gruppo Nazionale per L'Analisi Matematica, la Probabilit\`{a} e le loro Applicazioni (GNAMPA) of the Instituto Nazionale di Alta Matematica (INdAM). This paper was written partially during the stay of the first author at Tohoku University within the period October to December 2018. He thanks the Mathematical Institute of Tohoku University for the worm hospitality and the excellent working conditions during this period. The second author is partially supported by the Grant-in-Aid for Scientific Research (B)(No.18H01132).

\end{document}